\documentclass[a4paper, 12pt,  reqno]{amsart}
\usepackage{amsmath,amssymb,amscd,amsthm}

\usepackage{color}
\usepackage{hyperref}

\usepackage{cases}

\usepackage{bm}
\usepackage{comment}

\allowdisplaybreaks[3]

\newtheorem{theorem}{Theorem} [section]

\newtheorem{lemma}[theorem]{Lemma}
\newtheorem{proposition}[theorem]{Proposition}
\newtheorem{remark}[theorem]{Remark}
\newtheorem{definition}[theorem]{Definition}
\newtheorem{corollary}[theorem]{Corollary}

\DeclareMathOperator*{\esssup}{ess \ sup}

\newcommand{\Z}{\mathbb{Z}}
\newcommand{\R}{\mathbb{R}}
\newcommand{\C}{\mathbb{C}}
\newcommand{\T}{\mathbb{T}}

\newcommand{\al}{\alpha}
\newcommand{\be}{\beta}
\newcommand{\dl}{\delta}

\newcommand{\Ta}{\Theta}

\newcommand{\eps}{\varepsilon}

\newcommand{\Ld}{\Lambda}

\newcommand{\ft}{\widehat}
\newcommand{\Ft}{{\mathcal{F}}}
\newcommand{\wt}{\widetilde}
\newcommand{\cj}{\overline}
\newcommand{\dx}{\partial_x}

\newcommand{\dt}{\partial_t}

\newcommand{\pd}{\partial}

\renewcommand{\l}{\ell}

\newcommand{\les}{\lesssim}
\newcommand{\ges}{\gtrsim}

\newcommand{\jb}[1]
{\langle #1 \rangle}

\newcommand{\Ml}{\mathcal{M}}
\newcommand{\Nl}{\mathcal{N}}
\newcommand{\Kl}{\mathcal{K}}

\newcommand{\N}{\mathbb{N}}

\newcommand{\fw}{\mathfrak w}

\newcommand{\fR}{\mathfrak R}

\newcommand{\I}{\mathrm{I}}
\newcommand{\II}{\mathrm{I \! I}}
\newcommand{\III}{\mathrm{I \! I \! I}}

\numberwithin{equation}{section}
\numberwithin{theorem}{section}

\let\Re=\undefined\DeclareMathOperator*{\Re}{Re}
\let\Im=\undefined\DeclareMathOperator*{\Im}{Im}

\begin{document}
\baselineskip = 14pt

\title[WP and IP for periodic semi-linear Schr\"odinger equations]
{Well- and ill-posedness of the Cauchy problem for semi-linear Schr\"odinger equations on the torus}

\author[T.~Kondo and M.~Okamoto]
{Toshiki Kondo and Mamoru Okamoto}

\address{
Toshiki Kondo\\
Department of Mathematics\\
Graduate School of Science\\ Osaka University\\
Toyonaka\\ Osaka\\ 560-0043\\ Japan
}
\email{u534463k@ecs.osaka-u.ac.jp}

\address{
Mamoru Okamoto\\
Department of Mathematics\\\
Graduate School of Advanced Science and Engineering\\
Hiroshima University\\
1-3-1 Kagamiyama,
Higashi-Hiroshima, 739-8526
Japan}
\email{mokamoto@hiroshima-u.ac.jp}

\subjclass[2020]{35Q55}

\keywords{Schr\"odinger equation; well-posedness; ill-posedness;
energy estimate;
gauge transformation}

\begin{abstract}
We consider the Cauchy problem for semi-linear Schr\"odinger equations on the torus $\T$.
We establish a necessary and sufficient condition on the polynomial nonlinearity for the Cauchy problem to be well-posed in the Sobolev space $H^s(\T)$ for $s>\frac 52$.
For the well-posedness,
we use the energy estimates and the gauge transformation.
For the ill-posedness,
we prove the non-existence of solutions to the Cauchy problem.
\end{abstract}

%
\maketitle
%


\section{Introduction}

\label{SEC:1}

We consider the Cauchy problem for the following semi-linear Schr\"odinger equation:
\begin{equation} 
\left\{
\begin{aligned}
&\dt u + i \dx^2 u = F(u, \dx u, \cj u, \cj{\dx u}), \\ 
&u|_{t=0} = \phi,
\end{aligned}
\right.
\label{NLS}
\end{equation}
where $\T := \R / 2\pi \Z$, $u = u(t,x) : \R \times \T \to \C$ is an unknown function,
and $\phi : \T \to \C$ is a given function.
Here, $F(\al, \be, \cj{\al}, \cj{\be})$ is a polynomial in $\al, \be, \cj \al, \cj \be$,
and $\cj \al$ denotes the complex conjugate of $\al \in \C$.

In the linear case,
the Cauchy problem
\[
\left\{
\begin{aligned}
&\dt u + i \dx^2 u = a(x)u + b(x) \dx u, \\ 
&u|_{t=0} = \phi
\end{aligned}
\right.
\]
is well-posed in $L^2(\T)$
if and only if
\[
\int_\T \Im b(x) dx =0.
\]
This is a torus version of the Mizohata condition.
See \cite{Chi02} and references therein.

We mention well-posedness results for nonlinear Schr\"odinger equations (NLS) on the torus.
When the nonlinearity involves a derivative,
its structure plays a crucial role in establishing well-posedness.
Indeed,
Gr\"unrock \cite{Gr00} proved that the Cauchy problem
\begin{equation}
\left\{
\begin{aligned}
&\dt u + i \dx^2 u = \dx (\cj u^m), \\ 
&u|_{t=0} = \phi
\end{aligned}
\right.
\label{NLDG1}
\end{equation}
is well-posed in $H^s(\T)$ for $s \ge 0$, $m \in \Z_{\ge 2}$, and $s>\frac 12 - \frac 1{m-1}$.
However,
Chihara \cite{Chi02} and Christ \cite{Chr03} proved that the Cauchy problem
\begin{equation}
\left\{
\begin{aligned}
&\dt u + i \dx^2 u = \dx (u^m), \\ 
&u|_{t=0} = \phi
\end{aligned}
\right.
\label{NLSC1}
\end{equation}
is ill-posed in $H^s(\T)$ for any $s \in \R$ and $m \in \Z_{\ge 2}$.
See also \cite{AmSi15,DNY21,Hao07,Her06,KNM23,KiTs18,KiTs23,KoOk,Tak99} and references therein
for periodic NLS with a derivative.
Quasi-linear Schr\"odinger equations on $\mathbb{T}$ are also studied.
Feola and Iandoli \cite{FeIa19} proved local well-posedness for quasi-linear Hamiltonian or parity-preserving Schr\"odinger equations.
See also \cite{FeIa21} for well-posedness results for fully nonlinear Schr\"odinger equations.
Moreover, Jeong and Oh \cite{JO23} proved the non-existence of solutions for quasi-linear Schr\"odinger equations that include the Hunter-Smothers equation.

When \eqref{NLS} is considered on $\R$,
the situation changes drastically.
Due to the local smoothing effect on $\R$,
it is possible to recover the derivative losses for \eqref{NLS} on $\R$
unless the nonlinearity includes derivative quadratic terms.
In fact,
Kenig, Ponce, and Vega \cite{KPV93} proved
that \eqref{NLS} is well-posed in $H^s(\R)$ for $s>\frac 72$
if $F$ is a cubic or higher-order polynomial.
On the other hand,
Christ \cite{Chr03} proved that
the Cauchy problem
\eqref{NLSC1} with $m=2$
is ill-posed in $H^s(\R)$ for any $s \in \R$.
See \cite{HGM22, KPV04, LiPo02, MMT21} for quasi-linear Schr\"odinger equations.

In this paper,
we establish a necessary and sufficient condition on $F$ for the well-posedness of \eqref{NLS} in $H^s(\T)$,
which is analogous to the Mizohata condition for linear equations.
Before stating the main result, we define some notations.

\begin{definition}
\label{def:sol}
Let $s>\frac 32$, $T > 0$, and $\phi \in H^s(\T)$.
We say that $u$ is a solution to \eqref{NLS} in $H^s(\T)$ on $[0,T]$
if $u$ satisfies the following two conditions:

\begin{enumerate}
\item
$u \in L^\infty ([0,T]; H^s(\T))$;

\item
For any $\chi \in C_c^\infty([0,T) \times \T)$,%
\footnote{Here, $C_c^\infty([0,T) \times \T)$ denotes the space of smooth functions with compact support in $[0,T) \times \T$.
}
we have
\begin{align*}
&- \int_0^T \int_\T u (t,x) \dt \chi(t,x) dx dt
- \int_\T \phi (x) \chi (0,x) dx
\\
&\quad
+ i \int_0^T \int_\T u(t,x) \dx^2 \chi(t,x) dx dt
\\
&=
\int_0^T \int_\T F(u(t,x), \dx u(t,x), \cj{u(t,x)}, \cj{\dx u(t,x)}) \chi (t,x) dx dt.
\end{align*}
\end{enumerate}
A solution on $[-T,0]$ is defined in a similar manner.
\end{definition}

The second condition (ii) in Definition \ref{def:sol} means that $u$ satisfies \eqref{NLS} in the sense of distribution.
The condition $s>\frac 32$, together with Proposition \ref{prop:bili2} below, ensures that the nonlinear term $F( u, \dx u, \cj u , \cj{\dx u})$ is well-defined.

\begin{remark}
\label{rem:sold}
\rm
If $s >\frac 32$ and $u \in C([0,T]; H^s(\T))$ is a solution to \eqref{NLS},
Proposition \ref{prop:bili2} below yields that
\[
F( u, \dx u, \cj u , \cj{\dx u}) \in C([0,T]; H^{s-1}(\T)).
\]
Hence,
$u \in C^1([0,T]; H^{s-2}(\T))$
and
\[
\dt u + i \dx^2 u = F( u, \dx u, \cj u , \cj{\dx u})
\]
holds in $H^{s-2}(\T)$ for $t \in [0,T]$.
\end{remark}

We say that \eqref{NLS} is (locally in time) well-posed in $H^s(\T)$ if the following two conditions are satisfied:
\begin{enumerate}
\item
For any $\phi \in H^s(\T)$,
there exist $T>0$ and a unique solution $u \in C([-T,T]; H^s(\T))$ to \eqref{NLS}.%
\footnote{In general, uniqueness is required in a weaker sense as follows:
There exists a function space $X \subset C([-T,T]; H^s(\T))$ such that a solution $u \in X$ is unique.
Namely, the uniqueness does not need to hold in the whole space $C([-T,T]; H^s(\T))$.
However, 
by Corollary \ref{cor:diffHs-1s} below,
if a solution $u \in C([-T,T]; H^s(\T))$ for $s>\frac 52$ exists,
the uniqueness in the whole space holds in our setting.}

\item
If $\phi_n$ converges to $\phi$ in $H^s(\T)$,
then the corresponding solution $u_n$ to \eqref{NLS} with the initial data $\phi_n$ converges to $u$ in $C([-T,T]; H^s(\T))$.
\end{enumerate}

For $\al = a+ib \in \C$ ($a,b \in \R$),
we write
\[
\frac{\pd F}{\pd \al} := \frac 12 \Big( \frac{\pd F}{\pd a} - i \frac {\pd F}{\pd b} \Big), \quad
\frac{\pd F}{\pd \cj{\al}} := \frac 12 \Big( \frac{\pd F}{\pd a} + i \frac {\pd F}{\pd b} \Big).
\]
We also use the following abbreviations:
\[
F_\al = \frac{\pd F}{\pd \al}, \quad
F_\be = \frac{\pd F}{\pd \be}, \quad
F_{\cj{\al}} = \frac{\pd F}{\pd \cj{\al}}, \quad
\dots. 
\]
The following is the main result in the present paper.

\begin{theorem}
\label{thm:equiv}
Let $s > \frac 52$.
Then, the Cauchy problem \eqref{NLS} is well-posed in $H^s(\T)$ if and only if
\begin{equation}
\int_\T \Im F_\be ( \psi(x),\dx \psi(x),\cj{\psi(x)},\cj{\dx \psi(x)} ) dx = 0
\label{Fbez}
\end{equation}
for any $\psi \in H^s(\T)$.
\end{theorem}

For the ill-posedness, we prove the non-existence of solutions to \eqref{NLS} as follows.

\begin{theorem}
\label{thm:nonexi}
Let $s>\frac 52$.
Suppose that there exists $\psi \in H^s (\T)$ such that
\begin{equation}
\int_\T \Im F_\be ( \psi(x),\dx \psi(x),\cj{\psi(x)},\cj{\dx \psi(x)} ) dx \neq 0.
\label{Fbez2}
\end{equation}
Then,
there exists $\phi  \in H^s(\T)$ such that for no $T > 0$
the Cauchy problem \eqref{NLS} has a solution in $C([0,T];H^s(\T))$ or $C([-T,0];H^s(\T))$.
\end{theorem}

Here, we give examples of nonlinearities:
\begin{enumerate}
\item[(a)]
The nonlinearity in \eqref{NLDG1} corresponds to
$F(\al,\be, \cj \al, \cj \be) = m \cj \al^{m-1} \cj \be$.
Since $F_\be (\al,\be, \cj \al, \cj \be) = 0$,
the condition \eqref{Fbez} holds.

\item[(b)]
The derivative NLS
\[
\dt u + i \dx^2 u = \dx (|u|^2 u)
\]
corresponds to $F(\al, \be, \cj \al, \cj \be) = 2 |\al|^2 \be + \al^2 \cj \be$.
Since
$F_\be(\al, \be, \cj \al, \cj \be) = 2|\al|^2$,
the condition \eqref{Fbez} holds.

If we put $i$ in front of the nonlinearty,
the structure changes.
Indeed,
\[
\dt u + i \dx^2 u
=
i \dx (|u|^2 u)
\]
corresponds to $F(\al, \be, \cj \al, \cj \be) = i (2 |\al|^2 \be + \al^2 \cj \be)$.
Since
$F_\be(\al, \be, \cj \al, \cj \be) = 2i |\al|^2$,
the condition \eqref{Fbez2} holds with $\psi(x)=1$.

\item[(c)]
The nonlinearity in \eqref{NLSC1} corresponds to $F(\al,\be, \cj \al, \cj \be) = m \al^{m-1} \be$.
Since $F_\be (\al,\be, \cj \al, \cj \be) = m \al^{m-1}$,
the condition \eqref{Fbez2} holds with $\psi (x)= e^{\frac \pi{2 (m-1)}i}$.
Hence, Theorem \ref{thm:nonexi} applies to \eqref{NLSC1}.
In particular, it implies that the flow map is NOT well-defined in $H^s(\T)$ for $s>\frac 52$.
\end{enumerate}

The unique existence of a solution to \eqref{NLS} was already proved in \cite{Chi02}.
In fact,
Chihara \cite{Chi02} proved the existence of a unique solution to \eqref{NLS} under the following condition:%
\footnote{Chihara treated not only for polynomials $F$ but also for smooth functions.
See also Remark \ref{rem:notPol} below for a comment regarding the case when $F$ is not a polynomial.
}
There exists a smooth real-valued function $\Phi (\psi, \cj \psi)$ such that
\[
\dx (\Phi (\psi, \cj{\psi}))
= \Im F_\be (\psi, \dx \psi, \cj{\psi}, \cj{\dx \psi})
\]
for any $\psi \in C^1(\T)$.
This condition is equivalent to
\eqref{Fbez} for any $\psi \in C^1(\T)$.
Indeed,
if $F$ satisfies \eqref{Fbez},
we can choose $\Phi$ as
$\dx^{-1} \Im F_\be (\psi, \dx \psi, \cj{\psi}, \cj{\dx \psi})$,
where the operator $\dx^{-1}$ is defined in \eqref{defande} below.
However,
Chihara showed that the continuous dependence of the initial data was in a weaker sense.
Specifically,
he showed that the sequence $\{ u_n \}_{n \in \N}$ of solutions to \eqref{NLS} converges to $u$ in $C([-T,T]; H^{s'}(\T))$ for any $s'<s$.
In contrast,
for the well-posedness result in Theorem \ref{thm:equiv},
we prove that the sequence $\{ u_n \}_{n \in \N}$ converges to $u$ in $C([-T,T]; H^s(\T))$.

\begin{remark}
\rm
By an approximation argument,
the condition \eqref{Fbez} for ``any $\psi \in H^s(\T)$ when $s>\frac 52$'' is equivalent to that for ``any $\psi \in C^\infty(\T)$''.
In particular, our condition is equivalent to the one in Chihara's paper \cite{Chi02}.
\end{remark}

For the well-posedness of \eqref{NLS} in $H^s(\T)$,
we apply a Bona-Smith-type approximation argument as in \cite{BoSm75}.
To ensure the existence of approximate smooth solutions on the time interval $[-T,T]$,
where $T$ depends on $\| \phi \|_{H^s}$ but is independent of the approximation,
we rely on the persistence of regularity.
See Remark \ref{rem:pers} for the persistence.
Moreover,
establishing the persistence of regularity
requires energy estimates not only in $H^s(\T)$ but also in $H^r(\T)$ for some $r>s$.
Accordingly,
we derive such energy estimates in this paper,
while the existence part of the solution follows the approach in \cite{Chi02}.

To reveal the most troublesome term in the nonlinearity,
we consider the system of $u$, $v := \dx u$, and $w := \dx^2 u$.
The equation for $w$ is given by
\begin{align*}
\dt w + i \dx^2 w
&=
F_\be (u,v, \cj u, \cj v) \dx w
+ F_{\cj \be} (u,v, \cj u, \cj v) \cj{\dx w}
\\
&\quad
+ \text{(a polynomial in $u,v,w, \cj u, \cj v, \cj w$)}.
\end{align*}
We can not directly apply an energy estimate to the first term on the right-hand side.
In contrast, the second term is less problematic.
Thus,
we employ a gauge transformation to cancel out the troublesome term.
See, for example,
\cite{HaOz92,Her06,Tak99}
on gauge transformations.
However, since the gauge transformation can not eliminate the constant term,
we use the condition \eqref{Fbez} to ensure that the constant vanishes.
Specifically,
by setting
\[
W := \exp \Big( - \frac i2 \dx^{-1} F_\be (u,v,\cj u, \cj v) \Big) w,
\]
the equation for $W$ becomes
\begin{equation}
\dt W + i \dx^2 W
=
\int_\T  F_\be (u,v, \cj u, \cj v) dx
\cdot \dx W
+ \text{(other terms)}.
\label{eq:Wgg}
\end{equation}
Using \eqref{Fbez}, we can apply the energy estimate to this equation.

To apply the energy method for $W$,
we require the lower bound $\frac 52 = 2 + \frac 12$ for the regularity in Theorem \ref{thm:equiv}.
Although this regularity threshold is not optimal (see \cite{Gr00}, for example),
our main goal in this paper is to establish the nonlinear version of the Mizohata condition.
Therefore, we do not pursue optimal regularity here.

If \eqref{Fbez} does not hold,
the term $i \int_\T \Im F_\be (u,v, \cj u ,\cj v) dx$ remains
in the first part on the right-hand side of \eqref{eq:Wgg}.
Then,
roughly speaking,
the Cauchy-Riemann-type elliptic operator $\dt + i a \dx$ (for some $a \in \R \setminus \{ 0 \}$) appears in \eqref{eq:Wgg}.
This observation suggests the ill-posedness of \eqref{NLS}.

For the proof of the ill-posedness,
we apply integration by parts, as in \cite{KiTs18}, combined with the gauge transformation.
Moreover,
we give a necessary condition on the initial data for the existence of a solution.
See Theorem \ref{thm:NE2} below.
Using an approximation argument,
we then construct $\phi \in H^s(\T)$ such that
no corresponding solution to \eqref{NLS} exists when \eqref{Fbez2} holds.

Since the Cauchy-Riemann-type elliptic operator is of first order
and ``(other terms)'' in \eqref{eq:Wgg} contain nonlinear terms with a derivative,
we rely on the dispersive nature in the proof of the ill-posedness.
See also Remark 2.8 in \cite{KiTs18}.
In particular, we need to show that
the dispersive effect does not vanish
for the nonlinear interactions involving derivative nonlinear terms in  ``(other terms)'' of \eqref{eq:Wgg}.

\begin{remark}
\label{rem:initialdata1}
\rm
In Theorem \ref{thm:nonexi},
we can not take $\phi =\psi$ in general.
Indeed,
when $F(\al, \be, \cj \al, \cj \be) = (a_0 + a_1 \al + a_2 \cj \al) \be$ for some $a_0,a_1,a_2 \in \C$,
the corresponding \eqref{NLS} is
\[
\left\{
\begin{aligned}
&\dt u + i \dx^2 u = (a_0+a_1 u + a_2 \cj u) \dx u, \\ 
&u|_{t=0} = \phi.
\end{aligned}
\right.
\]
In this case,
any constant function is a solution.
Here,
the condition \eqref{Fbez2} becomes
\[
0 \neq
\int_\T \Im (a_0 + a_1 \psi + a_2 \cj{\psi}) dx
= \Im \int_\T \big( a_0 + (a_1-\cj{a_2}) \psi \big) dx.
\]
Let $c$ be a constant such that $\Im (a_0 + (a_1-\cj{a_2}) c) \neq 0$.
Then,
$\psi (x) =c$ satisfies \eqref{Fbez2},
but $u(t,x) = \psi(x) = c$ is a solution.
\end{remark}

Recently,
Tsugawa \cite{Tsu} independently obtained results similar to those in Theorem \ref{thm:equiv}.
In the proof of the well-posedness part,
he mainly used energy estimates for linear Schr\"odinger equations with variable coefficients.
In contrast, we directly apply the energy method to \eqref{NLS} and its parabolic regularized equations.

Additionally, Tsugawa \cite{Tsu} used an energy estimate to derive a parabolic smoothing effect,
which leads to the non-existence result.
See also \cite{Roy22, TaTs22, Tsu16, Tsu17a} for parabolic smoothing effects in other dispersive equations.
As mentioned above,
we utilize integration by parts, as in \cite{KiTs18}, in conjunction with the gauge transformation to obtain the non-existence result.
Specifically, the ill-posedness result stems from the presence of a problematic term after applying a gauge transformation.
Similar scenarios have been studied for KdV-type equations \cite{Akh14, AAW19, AmWr13, Roy22}.
The ill-posedness results follow from proving that we can handle terms other than the problematic part.
As noted before Remark \ref{rem:initialdata1},
we reply on the nonlinear smoothing effect to manage these terms in this paper.

\begin{remark}
\rm
Even if \eqref{NLS} is ill-posed in $H^s(\T)$,
replacing $H^s(\T)$ with another function space
might result in \eqref{NLS} being well-posed in that space.
For instance,
Chung et al. \cite{CGKO17} proved the small-data global well-posedness for 
\eqref{NLSC1} with $m=2$
in the Sobolev space under the mean-zero assumption
\[
\bigg\{ f \in H^s(\T) \mid \int_\T f dx  =0 \bigg\}
\]
for $s \ge 0$.
Moreover,
when $F$ is independent of $\cj \al$ and $\cj \be$,
Nakanishi and Wang \cite{NaWa} proved that
\eqref{NLS} is globally well-posed in a space of distributions whose Fourier support is in the half-space.

As suggested by these results,
given $F$,
it may be possible to choose an appropriate function space $X$ in which \eqref{NLS} is well-posed.
However,
since the choice of such $X$ heavily depends on the structure of $F$,
we do not address this issue in this paper.
\end{remark}

\begin{remark}
\label{rem:notPol}
\rm
To avoid some technical difficulties,
such as composition estimates for fractional derivatives and the regularity of smoothly approximated solutions,
we restrict our consideration to polynomial $F$.
However,
when focusing solely on the well-posedness in $H^3(\T)$,
our proof remains valid when $F$ is at least of class $C^7$.
Indeed, under this assumption,
for $\phi \in H^3(\T)$,
the regularized equation \eqref{RNLS} has the solution
$u^\eps \in C([0,T(\eps)]; H^3(\T)) \cap C ((0,T(\eps)]; H^6(\T))$.
Here, we recall that, given $s \in \R$ and $T>0$,
\[
\dt \| f(t) \|_{H^s}^2
= 2 \Re \big( \jb{\dx}^{s-1} \dt f(t), \jb{\dx}^{s+1} f(t) \big)_{L^2}
\]
holds for $t \in (0,T)$,
when $f \in C((0,T); H^{s+1}(\T)) \cap C^1((0,T); H^{s-1}(\T))$.
Thus,
as shown in Propositions \ref{prop:dtEne} and \ref{prop:diffHs-1} below,
the energy estimates for $\| u^\eps (t) \|_{H^5}^2$ and $\| u^{\eps_1} (t)- u^{\eps_2} (t) \|_{H^4}^2$ hold.
Consequently,
there exist $T>0$ and a solution $u \in L^\infty([0,T]; H^3(\T))$ to \eqref{NLS}.
Moreover, the persistence of regularity holds:
if $\phi \in H^5(\T)$, then $u \in C([0,T]; H^4(\T)) \cap C^1([0,T]; H^2(\T))$.
Since
this regularity is sufficient to yield Proposition \ref{prop:diffNM} with $s=3$,
our argument proceeds as expected.
\end{remark}

The plan of this paper is as follows.
In Section \ref{Sec:bilin},
we give bilinear estimates needed for the proof of Theorem \ref{thm:equiv}.
In Section \ref{Sec:WP},
we prove the well-posedness part of Theorem \ref{thm:equiv} by using energy estimates.
In Section \ref{Sec:NE},
we prove Theorem \ref{thm:nonexi}, which is the ill-posedness part of Theorem \ref{thm:equiv}.

We conclude this section with notations.
When $f \in L^1(\T)$, we set
\begin{align*}
\int_\T f dx
&:=
\frac 1{2\pi} \int_{-\pi}^\pi f(x) dx,
\\
\ft f(k)
&:= \int_\T f(x) e^{-ikx} dx
\end{align*}
for $k \in \Z$.
We also denote by $\ft f(k)$ or $\Ft [f](k)$ the Fourier coefficient of a distribution $f$ on $\T$.
Moreover, we define
\[
\jb{\dx}^s f (x) := \sum_{k \in \Z} \jb{k}^s \ft f(k) e^{ikx}.
\]
for $s \in \R$,
where $\jb{k} := \sqrt{1+|k|^2}$.

For a distribution $f$ on $\T$,
we define
\begin{align}
P_0 f
&:=
\ft f(0),
\label{defP0}
\qquad
P_{\neq 0} f
:=
f - P_0 f,
\\
\dx^{-1} f (x)
&:=
\sum_{k \in \Z \setminus \{0\}} \frac1{ik} \ft f(k) 
e^{ikx}.
\label{defande}
\end{align}
Note that $\dx (\dx^{-1} f) = \dx^{-1} (\dx f) = P_{\neq 0} f$.
Moreover, we also define the operators
\begin{align*}
P_+ f (x) &:= \sum_{\substack{k \in \N}} \ft f(k) e^{ikx},
&
P_- f (x) &:= \sum_{\substack{k \in \N}} \ft f(-k) e^{-ikx},
\end{align*}
where $\N$ denotes the set of positive integers.

We use $A \les B$ to denote $A \le C B$ with a constant $C>0$.
We also use $A \sim B$ to denote $A \les B$ and $A \ges B$.

\section{Bilinear estimates}
\label{Sec:bilin}

We use the following bilinear estimate.

\begin{proposition}
\label{prop:bili}
Let $s_0,s_1,s_2 \in \R$ satisfy
\[
\min (s_0+s_1, s_1+s_2, s_2+s_0) \ge 0 , \quad
s_0+s_1+s_2 > \frac 12
\]
or
\[
\min (s_0+s_1, s_1+s_2, s_2+s_0) > 0 , \quad
s_0+s_1+s_2 \ge \frac 12.
\]
Then, we have
\[
\| fg \|_{H^{-s_0}}
\les \| f \|_{H^{s_1}} \| g \|_{H^{s_2}}
\]
for $f \in H^{s_1}(\T)$ and $g \in H^{s_2}(\T)$.
\end{proposition}

See Corollary 3.16 on page 855 in \cite{Tao01}.
While the bilinear estimate on $\R$ is proved in \cite{Tao01},
Proposition \ref{prop:bili} follows from a straightforward modification of the proof.

Moreover, the following bilinear estimate holds true.
See Lemma B.1 in \cite{II01} for example.

\begin{proposition}
\label{prop:bili2}
Let $s \ge 0$ and $r>\frac 12$.
Then, we have
\[
\| fg \|_{H^s}
\les \| f \|_{H^s} \| g \|_{H^r} + \| f \|_{H^r} \| g \|_{H^s}
\]
for any $f,g \in H^s(\T) \cap H^r(\T)$.
\end{proposition}

We also use the following bilinear estimate.

\begin{lemma}
\label{lem:bili3}
Let $s \ge 0$ and $r>\frac 12$.
Then, we have
\[
\| P_+(f P_- g) \|_{H^s} + \| P_-(f P_+ g) \|_{H^s}
\les \| f \|_{H^s} \| g \|_{H^r}
\]
for any $f \in H^s(\T)$ and $g \in H^r(\T)$.
\end{lemma}

\begin{proof}
We only consider the estimate for $\| P_+(f P_- g) \|_{H^s}$.
For $n \in \N$,
H\"older's inequality yields that
\begin{align*}
|\Ft[ P_+(f P_- g)] (n)|
&= \Big| \sum_{k \in \N} \ft f(n+k) \ft g(-k) \Big|
\\
&\le
\Big( \sum_{k \in \N} \jb{k}^{-2r} |\ft f(n+k)|^2 \Big)^{\frac 12}
\| g \|_{H^r}
.
\end{align*}
Note that $n \le n+k$ holds for $n,k \in \N$.
It follows from $s \ge 0$ and $r>\frac 12$ that
\begin{align*}
\| P_+(f P_- g) \|_{H^s}
&\le
\Big( \sum_{n \in \N} \sum_{k \in \N} \jb{k}^{-2r} \jb{n+k}^{2s} |\ft f(n+k)|^2 \Big)^{\frac 12}
\| g \|_{H^r}
\\
&\les
\| f \|_{H^s} \| g \|_{H^r}.
\end{align*}
This concludes the proof.
\end{proof}

To obtain the well-posedness in $H^s(\T)$ for $s>\frac 52$,
we use a commutator estimate.
Let $s \ge 0$ and $f \in H^{s+1}(\T)$.
We define the operator $[ \jb{\dx}^s, f] \dx : H^{s+1}(\T) \to L^2(\T)$ by
\[
[ \jb{\dx}^s, f] \dx g
:= \jb{\dx}^s (f \dx g) - f \jb{\dx}^s \dx g
\]
for $g \in H^{s+1} (\T)$.
This operator is continuously (uniquely) extended as follows.

\begin{proposition}
\label{prop:Com1}
Let $s \ge 0$ and $\eps>0$.
Then,
we have
\begin{align*}
\| [ \jb{\dx}^s, f] \dx g \|_{L^2}
&\les
\| f \|_{H^{\frac 32+\eps}} \| g \|_{H^s}
+ \| f \|_{H^{s+1}} \| g \|_{H^{\frac 12+\eps}}
\end{align*}
for any $f \in H^{\max (s+1,\frac 32+\eps)}(\T)$ and $g \in H^{\max(s, \frac 12+\eps)}(\T)$.
\end{proposition}

\begin{proof}
By a density argument,
it suffices to prove the estimate for $f \in H^{\max (s+1,\frac 32+\eps)}(\T)$ and $g \in H^{s+1}(\T)$.
A direct calculation shows that
\begin{equation}
\begin{aligned}
&\Ft \big[ \jb{\dx}^s (f \dx g) - f \jb{\dx}^s \dx g \big] (n)
\\
&= \sum_{k \in \Z} \big( \jb{n}^s  - \jb{k}^s ) ik \ft f(n-k)\ft g(k) \\
&=
\bigg( \sum_{\substack{k \in \Z \\ |k| \ge 2|n|}}
+ \sum_{\substack{k \in \Z \\ |k|< \frac{|n|}2}}
+ \sum_{\substack{k \in \Z \\ \frac{|n|}2 \le |k|< 2|n|}} \bigg)
\big( \jb{n}^s  - \jb{k}^s ) ik \ft f(n-k)\ft g(k)
\\
&=: \I (n) + \II (n) + \III (n)
\end{aligned}
\label{dec1}
\end{equation}
for $n \in \Z$.
We treat these three cases separately.

When $|k| \ge 2|n|$,
it holds that
$\jb{n-k} \sim \jb{k} \ges \jb{n}$.
With $s \ge 0$ and $\eps>0$,
we have
\begin{align*}
|\I (n)|
&\les
\sum_{\substack{k \in \Z \\ |k| \ge 2|n|}}
\jb{n-k} |\ft f(n-k)| \jb{k}^s |\ft g(k)|\\
&\les
\bigg( \sum_{\substack{k \in \Z \\ |k| \ge 2|n|}}
\jb{n-k}^2 |\ft f(n-k)|^2 \bigg)^{\frac 12}
\| g \|_{H^s}
\\
&\les
\jb{n}^{-\frac 12-\eps}
\| f \|_{H^{\frac 32+\eps}}
\| g \|_{H^s}.
\end{align*}
This yields that
\begin{equation}
\bigg( \sum_{n \in \Z} |\I(n)|^2 \bigg)^{\frac 12}
\les
\| f \|_{H^{\frac 32+\eps}}
\| g \|_{H^s}.
\label{In1}
\end{equation}

When $|k|< \frac{|n|}2$,
it holds that
$\jb{n} \sim \jb{n-k} \ges \jb{k}$.
With $s \ge 0$ and $\eps>0$,
we have
\begin{align*}
|\II (n)|
&\les
\sum_{\substack{k \in \Z \\ |k| < \frac{|n|}2}}
\jb{n-k}^s |\ft f(n-k)| \jb{k} |\ft g(k)|\\
&\les
\bigg( \sum_{\substack{k \in \Z \\ |k| < \frac{|n|}2}}
\jb{n-k}^{2s+2-2 \min (\frac 12+\eps,1)} |\ft f(n-k)|^2 \bigg)^{\frac 12}
\| g \|_{H^{\min (\frac 12+\eps,1)}}
\\
&\les
\jb{n}^{-\min (\frac 12+\eps,1)}
\| f \|_{H^{s+1}}
\| g \|_{H^{\frac 12+\eps}}.
\end{align*}
A similar calculation as in \eqref{In1} shows that
\begin{equation}
\bigg( \sum_{n \in \Z} |\II(n)|^2 \bigg)^{\frac 12}
\les
\| f \|_{H^{s+1}} \| g \|_{H^{\frac 12+\eps}}.
\label{IIn1}
\end{equation}

When
$\frac{|n|}2 \le |k| < 2|n|$,
it holds that
\begin{align*}
\big| \jb{n}^s - \jb{k}^s \big|
\les \jb{n-k} \jb{k}^{s-1}.
\end{align*}
Indeed,
if $|n-k| \ge \frac{|k|}2$,
the inequality follow from $|n-k| \sim |n| \sim |k|$.
If $|n-k| < \frac{|k|}2$,
the fundamental theorem of calculus yields that
\begin{align*}
\big| \jb{n}^s - \jb{k}^s \big|
&= \bigg| (n-k) \int_0^1 s \jb{k + \theta (n-k)}^{s-2} (k + \theta (n-k)) d\theta \bigg|
\\
&\les \jb{n-k} \jb{k}^{s-1}.
\end{align*}
Hence,
we have
\begin{align*}
|\III (n)|
&\les
\sum_{\substack{k \in \Z \\ \frac{|n|}2 \le |k| < 2|n|}}
\jb{n-k} |\ft f(n-k)| \jb{k}^s |\ft g(k)|\\
&=
\sum_{\substack{\l \in \Z \\ \frac{|n|}2 \le |n-\l| < 2|n|}}
\jb{\l} |\ft f(\l)| \jb{n-\l}^s |\ft g(n-\l)|.
\end{align*}
By Minkowski's integral inequality (with counting measures),
we have
\begin{equation}
\bigg( \sum_{n \in \Z} |\III(n)|^2 \bigg)^{\frac 12}
\les
\sum_{\l \in \Z}
\jb{\l} |\ft f(\l)|
\cdot
\| g \|_{H^s}
\les
\| f \|_{H^{\frac 32+\eps}}
\| g \|_{H^s}
\label{IIIn1}
\end{equation}
for $\eps>0$.
The desired estimate follows from
\eqref{dec1}--\eqref{IIIn1}.
\end{proof}

\begin{remark}
\label{rem:com1}
\rm
Lemma B.3 in \cite{II01} yields that
\[
\| [ \jb{\dx}^s, f] g \|_{L^2}
\les
\| f \|_{H^{\frac 32+\eps}} \| g \|_{H^{s-1}}
+ \| f \|_{H^s} \| g \|_{H^{\frac 12+\eps}}
\]
for $s \ge 1$, $\eps>0$, $f \in H^{\max (s,\frac 32+\eps)}(\T)$, and $g \in H^{\max(s-1, \frac 12+\eps)}(\T)$.
If we naively use this bound instead of Proposition \ref{prop:Com1},
we need to replace the assumption ``$s> \frac 52$'' in Theorem \ref{thm:equiv} with ``$s > \frac 72$''.
See \eqref{est2b} below, for example.
\end{remark}

\section{Well-posedness}
\label{Sec:WP}

In this section,
we prove the well-posedness for \eqref{NLS} under the condition \eqref{Fbez}.
Without loss of generality,
we may consider a solution to \eqref{NLS} forward in time.
Indeed,
if $u$ is a solution to \eqref{NLS} with \eqref{Fbez}, then
\[
u'(t,x) :=
\cj{u(-t,x)}
\]
solves
\[
\dt u' + i \dx^2 u' = - \cj{F(\cj{u'}, \cj{\dx u'}, u', \dx u')}.
\]
Since
\[
\frac{\pd}{\pd \be} \cj{F(\cj \al, \cj \be, \al, \be)}
= \cj{\frac{\pd F}{\pd \be} (\cj \al, \cj \be, \al, \be)},
\]
the nonlinear term of $u'$ also satisfies \eqref{Fbez}.

\subsection{Energy estimates}
\label{SUBSEC:energy}

We consider the Cauchy problem for the parabolic regularized \eqref{NLS}:
\begin{equation} 
\left\{
\begin{aligned}
&\dt u^{\eps} + (i-\eps) \dx^2 u^{\eps} = F(u^{\eps}, \dx u^{\eps}, \cj {u^{\eps}}, \cj{\dx u^{\eps}}), \\ 
&u^{\eps}(0,x) = \phi(x)
\end{aligned}
\right.
\label{RNLS}
\end{equation}
for $\eps>0$, $t>0$, and $x \in \T$.
Thanks to the parabolic regularization,
the existence of a solution to \eqref{RNLS} follows from the standard contraction argument.

\begin{proposition}
\label{prop:exRNLS}
Let $\eps>0$ and $s> \frac 52$.
Then, for any $\phi \in H^s(\T)$,
there exist $T(\eps)>0$ and a unique solution $u^\eps \in C([0,T(\eps)]; H^s(\T))$ to \eqref{RNLS}.
\end{proposition}

See Lemma 4.1 in \cite{Chi02} for example.
Since
$F$ is a polynomial,
$u^\eps$ belongs to $C^\infty ((0,T(\eps)) \times \T)$.

In this subsection,
we prove that there exist $T>0$ which is independent of $\eps \in (0,1)$ and a solution $u \in L^{\infty}([0,T];H^s(\T)) \cap C([0,T]; H^{s-1}(\T))$ to \eqref{NLS}.

Let $u^{\eps}$ be the solution to \eqref{RNLS} obtained in Proposition \ref{prop:exRNLS}.
Set
\[
v^{\eps}
=
\dx u^{\eps}.
\]
For simplicity, we write
\begin{align*}
F^{\eps} &= F(u^{\eps},v^{\eps}, \cj {u^{\eps}} ,\cj {v^{\eps}}), &
F_\al^{\eps} &= F_\al (u^{\eps},v^{\eps}, \cj {u^{\eps}} ,\cj {v^{\eps}}), &
& \ \dots.
\end{align*}
Then,
$u^\eps$ satisfies
\begin{equation}
\dt u^{\eps} + (i-\eps) \dx^2 u^\eps
= F^\eps .
\label{ueps}
\end{equation}
We also set
\[
w^{\eps}
=
\dx^2 u^{\eps}.
\]
Then,
$v^{\eps}$ satisfies
\begin{equation}
\begin{aligned}
\dt v^{\eps} + (i-\eps) \dx^2 v^{\eps}
&=
F^{\eps}_{\al} v^{\eps}
+
F^{\eps}_{\be} w^{\eps}
+
F^{\eps}_{\cj \al} \cj v
+
F^{\eps}_{\cj \be} \cj{w^{\eps}}
.
\end{aligned}
\label{veps}
\end{equation}
Moreover,
$w^{\eps}$ satisfies
\begin{equation}
\dt w^{\eps} + (i-\eps) \dx^2 w^{\eps}
=
F^{\eps}_{\be} \dx w^{\eps}
+
F^{\eps}_{\cj \be} \cj{\dx w^{\eps}}
+
R^{\eps},
\label{weps}
\end{equation}
where $R^\eps$ denotes a polynomial in
$u^{\eps},v^{\eps},w^{\eps},\cj{u^{\eps}},\cj{v^{\eps}},\cj{w^{\eps}}$.

To eliminate the problematic part in the first term on the right-hand side of \eqref{weps},
we apply a gauge transformation.
Define
\begin{align}
\Ld^{\eps}
&:=
-\frac {i}2 \dx^{-1} F_\be^{\eps},
\label{Ldep}
\\
W^{\eps}
&:=
e^{-\Ld^{\eps}} w^{\eps},
\label{Wga}
\end{align}
where $\dx^{-1}$ is defined in \eqref{defande}.

A direct calculation shows that
\begin{align*}
\dt W^{\eps}
&= e^{-\Ld^{\eps}} (- \dt \Ld^{\eps} \cdot w^{\eps} + \dt w^{\eps}),
\\
\dx W^{\eps}
&= e^{-\Ld^{\eps}} (- \dx \Ld^{\eps} \cdot w^{\eps} + \dx w^{\eps}),
\\
\dx^2 W^{\eps}
&=
e^{-\Ld^{\eps}} \big( (\dx \Ld^{\eps})^2 w^{\eps} -2 \dx \Ld^{\eps} \cdot \dx w^{\eps} - \dx^2 \Ld^{\eps} \cdot w^{\eps} + \dx^2 w^{\eps} \big).
\end{align*}
From \eqref{weps} and \eqref{Wga},
we obtain
\begin{equation}
\begin{aligned}
&\dt W^{\eps} + (i-\eps) \dx^2 W^{\eps}
\\
&=
(F_\be^{\eps} - 2(i-\eps) \dx \Ld^{\eps})
( \dx W^{\eps} + \dx \Ld^{\eps} \cdot W^{\eps})
\\
&\quad
+ F_{\cj \be}^\eps e^{-\Ld^{\eps} + \cj{\Ld^{\eps}}} \big( \cj{\dx W^{\eps}} + \cj{\dx \Ld^{\eps} \cdot W^{\eps}} \big)
+ e^{-\Ld^{\eps}} R^\eps
\\
&\quad
+
\big( - \dt \Ld^{\eps} + (i - \eps) (\dx \Ld^{\eps})^2 - (i - \eps) \dx^2 \Ld^{\eps} \big) W^{\eps}.
\end{aligned}
\label{NLS2eps}
\end{equation}

By \eqref{Ldep}, we have
\begin{equation}
\dx \Ld^{\eps} 
=-\frac i2 \dx \dx^{-1} F_\be^{\eps}
=-\frac i2 P_{\neq 0}F_\be^{\eps} .
\label{dxLdep}
\end{equation}
Moreover,
\begin{equation}
\dx^2 \Ld^{\eps}
=
-\frac i2 \dx F_\be^{\eps} 
=
-\frac i2
\big(
F_{\al \be}^{\eps} v^{\eps} + F_{\be \be}^{\eps} w^{\eps} + F_{\cj \al \be}^{\eps} \cj{v^{\eps}} + F_{\be \cj \be}^{\eps} \cj{w^{\eps}}
\big)
\label{dxxLdep}
\end{equation}
is a polynomial in $u^{\eps},v^{\eps},w^{\eps}, \cj{u^{\eps}}, \cj{v^{\eps}}, \cj{w^{\eps}}$.

It follows from \eqref{ueps} and \eqref{veps} that
\begin{align*}
\dt u^{\eps}
&= R_1^\eps, \\
\dt v^{\eps}
&= (-i + \eps) \dx w^{\eps}
+ R_2^\eps,
\label{uvep}
\end{align*}
where $R_1^{\eps}$ and $R_2^{\eps}$ are polynomials in $u^{\eps},v^{\eps},w^{\eps}, \cj u^{\eps}, \cj v^{\eps}, \cj w^{\eps}$.
A direct calculation shows that
\begin{align*}
\dt F_\be^{\eps}
&=F_{\al \be}^{\eps} \dt u^{\eps} + F_{\be \be}^{\eps} \dt v^{\eps} + F_{\cj \al \be}^{\eps} \cj{\dt u^{\eps}} + F_{\be \cj \be}^{\eps} \cj{\dt v^{\eps}}
\\
&=
(-i+ \eps) F_{\be \be}^{\eps} \dx w^{\eps}
+ (i+\eps) F_{\be \cj \be}^{\eps} \cj{\dx w^{\eps}}
\\
&\quad
+ F_{\al \be}^{\eps} R_1^\eps
+ F_{\cj \al \be}^{\eps} \cj{R_1^\eps}
+ F_{\be \be}^{\eps} R_2^\eps
+ F_{\be \cj \be}^{\eps} \cj{R_2^\eps}.
\end{align*}
Since
\begin{align*}
F_{\be \be}^{\eps} \dx w^{\eps}
&= \dx (F_{\be \be}^{\eps} w^{\eps})
- (F_{\al \be \be}^{\eps} v^{\eps} + F_{\be \be \be}^{\eps} w^{\eps} + F_{\cj \al \be \be}^{\eps} \cj{v^{\eps}} + F_{\be \be \cj \be}^{\eps} \cj{w^{\eps}} )w^{\eps},
\\
F_{\be \cj \be}^{\eps} \cj{\dx w^{\eps}}
&= \dx (F_{\be \cj \be}^{\eps} \cj{w^{\eps}})
- (F_{\al \be \cj \be}^{\eps} v^{\eps} + F_{\be \be \cj \be}^{\eps} w ^{\eps}+ F_{\cj \al \be \cj \be}^{\eps} \cj{v^{\eps}} + F_{\be \cj \be \cj \be}^{\eps} \cj{w^{\eps}} ) \cj{w^{\eps}},
\end{align*}
we can write
\begin{equation}
\dt F_\be^{\eps}
= 
(-i+\eps) \dx (F_{\be \be}^{\eps} w^{\eps})
+ (i+\eps) \dx (F_{\be \cj \be}^{\eps} \cj{w^{\eps}})
+ R_3^{\eps},
\label{Tabbep}
\end{equation}
where $R_3^{\eps}$ is a polynomial in $u^{\eps},v^{\eps},w^{\eps}, \cj u^{\eps}, \cj v^{\eps}, \cj w^{\eps}$.
From \eqref{Ldep} and \eqref{Tabbep},
we obtain that
\begin{equation}
\begin{aligned}
2i \dt \Ld^\eps
&=
(-i+\eps) P_{\neq 0} (F_{\be \be}^\eps w^\eps)
+ (i+\eps) P_{\neq 0} (F_{\be \cj \be}^\eps \cj{w^\eps})
+ \dx^{-1} R_3^\eps.
\end{aligned}
\label{intGep}
\end{equation}
It follows from \eqref{dxLdep}, \eqref{dxxLdep}, and \eqref{intGep} that 
\[
\dx \Ld^{\eps}, \qquad \dx^2 \Ld^{\eps}, \qquad \dt \Ld^{\eps}
\]
are expressed as
polynomials in $u^{\eps},v^{\eps},w^{\eps}, \cj u^{\eps}, \cj v^{\eps}, \cj w^{\eps}$ and as integrals of polynomials in $u^{\eps},v^{\eps},w^{\eps}, \cj u^{\eps}, \cj v^{\eps}, \cj w^{\eps}$.%
\footnote{
Note that \eqref{dxLdep} and each term on the right-hand side of \eqref{intGep} are not polynomials in $u^{\eps},v^{\eps},w^{\eps}, \cj u^{\eps}, \cj v^{\eps}, \cj w^{\eps}$.
In fact, since
\[
P_{\neq 0}F_\be^{\eps}
= F_\be^{\eps} - \int_\T F_\be^{\eps} dx
,
\quad
\dx^{-1} R_3^\eps (t,x)
= \int_\T \int_{y}^x P_{\neq 0} R_3^\eps (t,x') dx' dy,
\]
they require integration.
}

Therefore, we can rewrite \eqref{NLS2eps} as follows:
\begin{equation}
\begin{aligned}
\dt W^{\eps} + (i - \eps)\dx^2 W^{\eps}
&=
(P_0 F_\be^{\eps}) \dx W^{\eps} 
-
i \eps (P_{\neq 0} F_\be^{\eps}) \dx W^{\eps} 
\\
&\quad
+ 
F_{\cj \be}^{\eps} e^{- \Ld^{\eps} + \cj {\Ld ^{\eps}}} \cj {\dx W^{\eps}}
+
\fR^{\eps},
\label{Weps}
\end{aligned}
\end{equation}
where $\fR^{\eps}$
is expressed as
polynomials in $u^{\eps},v^{\eps},W^{\eps},\cj u^{\eps},\cj v^{\eps},\cj W^{\eps}, e^{\Ld^{\eps}}, e^{\cj{\Ld^{\eps}}}$ and as integrals of polynomials in $u^{\eps},v^{\eps},W^{\eps},\cj u^{\eps},\cj v^{\eps},\cj W^{\eps}, e^{\Ld^{\eps}}, e^{\cj{\Ld^{\eps}}}$.
Note that the coefficients of $\fR^\eps$ are bounded for $\eps \in (0,1)$.
In the following, for simplicity, we will refer to terms like $\fR^{\eps}$ as harmless.

Set
\begin{align}
E_s^\eps(t)
&:=
\sqrt{
\|u^{\eps}(t)\|_{H^{s-2}}^2
+
\|v^{\eps}(t)\|_{H^{s-2}}^2
+
\|W^{\eps}(t)\|_{H^{s-2}}^2}.
\label{Eesp}
\end{align}

\begin{lemma}
\label{lem:Enebd1}
Let $s,r \in \R$ satisfy
\begin{equation}
\frac 52 < s \le r.
\label{condsr0}
\end{equation}
Then,
there exist constants $C_0 ( \| u^\eps (t) \|_{H^{s-1}})$ and $C_0' ( \| u^\eps (t) \|_{H^{s}})$ such that
\begin{align}
\| F^\eps (t) \|_{H^{r-2}}
+ \big\| e^{\Ld^\eps(t)} \big\|_{H^{r-1}}
&\le C_0(\| u^\eps (t) \|_{H^{s-1}}) (1+ \| u^\eps (t) \|_{H^{r-1}}),
\label{enFLdeps1}
\\
E^\eps_r(t)
&\le C_0' (\| u^\eps(t)\|_{H^s}) (1+ \| u^\eps (t)\|_{H^r})
\label{bdEr00}
\end{align}
for $\eps \in (0,1)$ and $t \in (0,T(\eps))$.
Moreover,
there exists a constant $\wt C_0 ( E_s^\eps (t))$ such that
\begin{align*}
\| u^\eps(t) \|_{H^r}
&\le \wt C_0 (E_s^\eps(t)) (1+E^\eps_r(t))
\end{align*}
for $\eps \in (0,1)$ and $t \in (0,T(\eps))$.
\end{lemma}

\begin{proof}
For simplicity, we suppress the time dependence in this proof.
Proposition \ref{prop:bili2} implies that
\begin{align*}
\| F^\eps \|_{H^{r-2}}
&\le
C_0 ( \| u^\eps \|_{H^{s-1}})
\big( 1+ \| u^\eps \|_{H^{r-2}} + \| \dx u^\eps \|_{H^{r-2}} \big)
\\
&\le
C_0 ( \| u^\eps \|_{H^{s-1}})
( 1+ \| u^\eps \|_{H^{r-1}} ).
\end{align*}
Similarly,
with \eqref{Ldep},
we have that
\begin{align*}
\| \Ld^\eps \|_{H^{r-1}}
&\les
\| F^\eps_\be \|_{L^2} +  \| F^\eps_\be \|_{H^{r-2}}
\le
C_0 (\| u^\eps \|_{H^{s-1}})
(1+\| u^\eps \|_{H^{r-1}}),
\\
\big\| e^{\Ld^\eps} \big\|_{H^{r-1}}
&\le
\sum_{\l=0}^\infty \frac 1{\l!} \| (\Ld^\eps)^\l \|_{H^{r-1}}
\le
1+
\sum_{\l=1}^\infty \frac 1{\l!} C^\l  \| \Ld^\eps \|_{H^{s-1}}^{\l-1} \| \Ld^\eps \|_{H^{r-1}}
\\
&\le
C_0 (\| u^\eps \|_{H^{s-1}})
(1+\| u^\eps \|_{H^{r-1}}).
\end{align*}
Hence,
\eqref{enFLdeps1} holds.

Proposition \ref{prop:bili2}
with \eqref{Wga} and \eqref{enFLdeps1}
implies that
\begin{align*}
\| W^\eps \|_{H^{r-2}}
&=
\big\| e^{-\Ld^\eps} \dx^2 u^\eps \big\|_{H^{r-2}}
\\
&\le
\big\| e^{-\Ld^\eps} \big\|_{H^{r-2}} \| \dx^2u^\eps \|_{H^{s-2}}
+ \big\| e^{-\Ld^\eps} \big\|_{H^{s-2}} \| \dx^2 u^\eps \|_{H^{r-2}}
\\
&\le
C_0' (\| u^\eps \|_{H^{s}})
(1+\| u^\eps \|_{H^r}).
\end{align*}With \eqref{Eesp},
we have \eqref{bdEr00}.

Similarly,
it holds that
\begin{align*}
\| \dx^2 u^\eps \|_{H^{r-2}}
&=
\big\| e^{\Ld^\eps} W^\eps \big\|_{H^{r-2}}
\\
&
\les
\big\| e^{\Ld^\eps} \big\|_{H^{r-2}} \| W^\eps \|_{H^{s-2}}
+ \big\| e^{\Ld^\eps} \big\|_{H^{s-2}} \| W^\eps \|_{H^{r-2}}
\\
&\le
C_0 (\| u^\eps \|_{H^{s-1}}) (1+ \| u^\eps \|_{H^{r-1}}) \| W^\eps \|_{H^{s-2}}
\\
&\quad
+
C_0 (\| u^\eps \|_{H^{s-1}}) \| W^\eps \|_{H^{r-2}}
\\
&\le
C(E_s^\eps) (1+ E^\eps_r).
\end{align*}
This concludes the proof.
\end{proof}

We prove the following energy estimate.

\begin{proposition}
\label{prop:dtEne}
Let $s,r \in \R$ satisfy \eqref{condsr0}.
Assume that $F$ satisfies \eqref{Fbez}.
Then,
there exists 
a constant $C_1(E_s^\eps(t))>0$ such that
\[
\dt E^\eps_r(t)
\le
C_1(E_s^\eps(t))
(1+ E^\eps_r(t))
\]
for $\eps \in (0,1)$ and $t \in (0,T (\eps))$.
\end{proposition}

\begin{proof}
For simplicity, we suppress the time dependence in this proof.
From \eqref{ueps} and Lemma \ref{lem:Enebd1},
we have
\begin{equation}
\begin{aligned}
\dt \|u^{\eps} \|_{H^{r-2}}^2
&=
2 \Re (\dt u^\eps, u^\eps)_{H^{r-2}}
\\
&=
- 2 \eps \| \dx u^\eps \|_{H^{r-2}}^2
+ 2 \Re (F^\eps , u^\eps)_{H^{r-2}}
\\
&
\le
2 \|F^\eps \|_{H^{r-2}} \| u^\eps \|_{H^{r-2}}
\le
C(E_s^\eps) (1+E^\eps_r)
\| u^\eps \|_{H^{r-2}}.
\end{aligned}
\label{est:u1}
\end{equation}
Moreover, using \eqref{veps} and Lemma \ref{lem:Enebd1}, we have that
\begin{equation}
\begin{aligned}
\dt \|v^{\eps} \|_{r-2}^2
&\le
C(E_s^\eps) (1+E^\eps_r)
\| v^\eps \|_{H^{r-2}}.
\label{est:v1}
\end{aligned}
\end{equation}%

Finally, we consider the estimate for $\dt \| W^\eps \|_{H^{r-2}}^2$.
A direct calculation with \eqref{Weps} yields that
\begin{equation}
\begin{aligned}
\frac 12 \dt \| W^\eps \|_{H^{r-2}}^2
&= \Re ( \dt W^\eps, W^\eps)_{H^{r-2}}
\\
&=
- \eps \| \dx W^\eps \|_{H^{r-2}}^2
\\
&\quad
+\Re \big( (P_0 F_\be^{\eps}) \dx W^{\eps} , W^\eps \big)_{H^{r-2}}
\\
&\quad
+
\Re \big( -i \eps (P_{\neq 0} F_\be^{\eps}) \dx W^{\eps} , W^\eps)_{H^{r-2}}
\\
&\quad
+
\Re \big( F_{\cj \be}^{\eps} e^{- \Ld^{\eps} + \cj {\Ld^{\eps}}} \cj {\dx W^{\eps}}, W^\eps \big)_{H^{r-2}}
\\
&\quad
+
\Re (\fR^{\eps} , W^\eps)_{H^{r-2}}
\\
&=:
- \eps \| \dx W^\eps \|_{H^{r-2}}^2
+ \sum_{j=1}^4 \I_j.
\end{aligned}
\label{enWe}
\end{equation}
%
It follows from \eqref{Fbez} and \eqref{defP0} that
\[
P_0 F_\be^{\eps} = \Re P_0 F_\be^{\eps}.
\]
From integration by parts, we have
\begin{equation}
\begin{aligned}
\I_1
&=
\Re(P_0 F_\be^{\eps}) 
\Re \int_{\T} (\dx \jb{\dx}^{r-2} W^{\eps}) \cj {\jb{\dx}^{r-2} W^{\eps}} dx
\\
&=
\frac12 \Re(P_0 F_\be^{\eps})\int_{\T} \dx (| \jb{\dx}^{r-2} W^{\eps} |^2) dx
=
0.
\end{aligned}
\label{est1a}
\end{equation}
 
We decompose $\I_2$ as follows:
\begin{equation}
\begin{aligned}
\I_2
&=
\eps
\Im  \big( \big[ \jb{\dx}^{r-2},P_{\neq 0} F^{\eps}_\be \big] \dx W^{\eps}, \jb{\dx}^{r-2} W^{\eps} \big)_{L^2}
\\
&\quad
+
\eps
\Im \big( (P_{\neq 0} F^{\eps}_\be) \jb{\dx}^{r-2} \dx W^{\eps}, \jb{\dx}^{r-2} W^{\eps} \big)_{L^2}
\\
&=: \I_{2,1} + \I_{2,2} .
\label{est2a}
\end{aligned}
\end{equation}
It follows from
Proposition \ref{prop:Com1}, \eqref{condsr0}, and Lemma \ref{lem:Enebd1}
that
\begin{equation}
\begin{aligned}
|\I_{2,1}|
&\le
\big\| \big[ \jb{\dx}^{r-2},P_{\neq 0} F^{\eps}_\be \big] \dx W^{\eps} \big\|_{L^2}
\|W^{\eps}\|_{H^{r-2}}
\\
&\les
\big(
\|P_{\neq 0} F^{\eps}_\be\|_{H^{s-1}} \|W^{\eps}\|_{H^{r-2}}
\\
&\qquad
+ \|P_{\neq 0} F^{\eps}_\be\|_{H^{r-1}} \|W^{\eps}\|_{H^{s-2}}
\big)
\|W^{\eps}\|_{H^{r-2}}
\\
&\le
C(E_s^\eps) (1+E^\eps_r) \|W^{\eps}\|_{H^{r-2}}.
\end{aligned}
\label{est2b}
\end{equation}
A direct calculation yields that
\begin{equation}
\begin{aligned}
|\I_{2,2}|
&\le
\eps
\| P_{\neq 0} F^{\eps}_\be \|_{L^\infty}
\| \jb{\dx}^{r-2} \dx W^{\eps} \|_{L^2} \| \jb{\dx}^{r-2} W^{\eps} \|_{L^2}
\\
&\le
\frac \eps2 \| \dx W^{\eps} \|_{H^{r-2}}^2
+
\frac \eps2
\| P_{\neq 0} F^{\eps}_\be \|_{L^\infty}^2
\| W^{\eps} \|_{H^{r-2}}^2
\\
&\le
\frac \eps2 \| \dx W^{\eps} \|_{H^{r-2}}^2
+
C(E_s^\eps) (1+E^\eps_r) \|W^{\eps}\|_{H^{r-2}}.
\end{aligned}
\label{est2ca}
\end{equation}
From \eqref{est2a}--\eqref{est2ca},
we obtain that
\begin{equation}
\begin{aligned}
|\I_2 (t)|
&\le
\frac \eps2 \| \dx W^{\eps} \|_{H^{r-2}}^2
+
C(E_s^\eps) (1+E^\eps_r) \|W^{\eps}\|_{H^{r-2}}.
\label{est2}
\end{aligned}
\end{equation}

We decompose $\I_3$ as follows:
\begin{align*}
\I_3
&= \Re \big( \big[ \jb{\dx}^{r-2} ,F_{\cj \be}^{\eps} e^{- \Ld^{\eps} + \cj {\Ld^{\eps}}} \big] \cj {\dx W^{\eps}}, \jb{\dx}^{r-2} W^\eps \big)_{L^2}
\\
&\quad
+ \Re \big( F_{\cj \be}^{\eps} e^{- \Ld^{\eps} + \cj {\Ld^{\eps}}} \cj {\jb{\dx}^{r-2} \dx W^{\eps}}, \jb{\dx}^{r-2} W^\eps \big)_{L^2}
\\
&=:
\I_{3,1} + \I_{3,2}.
\end{align*}
As in \eqref{est2b}, we have
\[
|\I_{3,1}|
\le
C(E_s^\eps) (1+E^\eps_r) \|W^{\eps}\|_{H^{r-2}}
.
\]
It follows from Lemma \ref{lem:Enebd1} that
\begin{align*}
|\I_{3,2}|
&=
\frac 12 \bigg| \Re \int \dx \big( F_{\cj \be}^{\eps} e^{- \Ld^{\eps} + \cj {\Ld^{\eps}}} \big) \big( \cj{\jb{\dx}^{r-2} W^\eps}\big)^2 dx \bigg|
\\
&\le
\big\| \dx \big( F_{\cj \be}^{\eps} e^{- \Ld^{\eps} + \cj {\Ld^{\eps}}} \big) \big\|_{L^\infty}
\| W^\eps \|_{H^{r-2}}^2
\\
&\les
\big\| F_{\cj \be}^{\eps} e^{- \Ld^{\eps} + \cj {\Ld^{\eps}}} \big\|_{H^{s-1}}
\| W^\eps \|_{H^{r-2}}^2
\\
&\le
C(E_s^\eps) (1+E^\eps_r) \|W^{\eps}\|_{H^{r-2}}
.
\end{align*}
Hence, we obtain that
\begin{align}
|\I_3|
\le
C(E_s^\eps) (1+E^\eps_r) \|W^{\eps}\|_{H^{r-2}}
.\label{est31}
\end{align}

Since $\fR^{\eps}$ is harmless,
the same calculation yields as in \eqref{est:u1} that
\begin{equation}
|\I_4|
\les \| \fR^{\eps} \|_{H^{r-2}} \| W^\eps \|_{H^{r-2}}
\le
C(E_s^\eps) (1+E^\eps_r) \|W^{\eps}\|_{H^{r-2}}
.
\label{est5}
\end{equation}
Combining \eqref{enWe}, \eqref{est1a}, and \eqref{est2}--\eqref{est5}, we get
\begin{equation}
\dt \|W^{\eps}(t)\|_{r-2}^2
\le
C(E_s^\eps) (1+E^\eps_r) \|W^{\eps}\|_{H^{r-2}}
.
\label{est6a}
\end{equation}
The desired bound follows from \eqref{Eesp}, \eqref{est:u1}, \eqref{est:v1}, and \eqref{est6a}.
\end{proof}

We prove the existence of a solution to \eqref{NLS}.
It follows from \eqref{Eesp} that
\begin{equation}
K
:= E_s^\eps(0)
\label{Eesp0}
\end{equation}
is independent of $\eps>0$.
Define
\begin{equation}
T := \frac 1{C_1(2K)+1} \log \frac{1+2K}{1+K}.
\label{time1}
\end{equation}
Note that $T$ is independent of $\eps>0$.

Set
\[
T^{\eps}_{\ast}
:=
\sup \{ T>0 \mid E_s^\eps(t)\le 2K \text{ for } t \in [0,T]\}.
\]
Then, we have
\[
T^{\eps}_{\ast}\ge T.
\]
Indeed,
if $T^{\eps}_{\ast}<T$,
there exists $t_0 \in (0, T^{\eps}_{\ast})$ such that $E^\eps_s(t_0) = 2K$.
It follows from Proposition \ref{prop:dtEne} with $r=s$ and \eqref{time1} that
\[
E^\eps_s(t_0)
\le
( 1 + E^\eps_s(0)) e^{C_1(E_s^\eps(0)) t_0} -1
\le
( 1 + K) e^{C_1(2K) T} -1
<2K.
\]
This contradicts to $E^\eps_s (t_0)=2K$.

\begin{corollary}
\label{cor:enebdur}
In addition to the assumption of Proposition \ref{prop:dtEne},
we further assume that $\phi \in H^r(\T)$.
Then,
there exists a constant $C_2 (\| \phi \|_{H^s})>0$ such that
\[
\| u^\eps \|_{L_T^\infty H^r}
\le C_2 (\| \phi \|_{H^s}) (1+\| \phi \|_{H^r}),
\]
for $\eps \in (0,1)$,
where $T$ is defined in \eqref{time1}
and
\[
\| u^\eps \|_{L_T^\infty H^r}
:=
\esssup_{t \in [0,T]} 
\| u^\eps (t) \|_{H^r}
.
\]
\end{corollary}

\begin{proof}
Proposition \ref{prop:dtEne}, \eqref{Eesp0}, and \eqref{time1} imply that
\[
\sup_{t \in [0,T]}
E^\eps_r(t)
\le
( 1 + E^\eps_r(0)) e^{C_1(2K) T} -1
< 1 + 2 E^\eps_r(0).
\]
The desired bound follows from Lemma \ref{lem:Enebd1}.
\end{proof}

By Corollary \ref{cor:enebdur},
there exists $u \in L^{\infty}([0,T];H^s(\T))$
and a sequence $\{ \eps_j \}_{j \in \N} \subset (0,1)$ such that
\[
\lim_{j \to \infty} \eps_j =0
\]
and $\{ u^{\eps_j} \}_{j \in \N}$ converges weak$\ast$ to $u$ in $L^\infty ([0,T]; H^s(\T))$.
Note that $u$ is a solution to \eqref{NLS} in the sense of Definition \ref{def:sol}.
Moreover,
if $\phi \in H^r(\T)$ for $r$ satisfying \eqref{condsr0},
by taking a further subsequence $\{ \eps_j' \}_{j \in \N} \subset \{ \eps_j \}_{j \in \N}$,
$\{ u^{\eps_j'} \}_{j \in \N}$ converges weak$\ast$ to $u$ in $L^\infty ([0,T]; H^r(\T))$.
Especially,
we have $u \in L^\infty([0,T]; H^r(\T))$.

Next, we prove that $u$ belongs to $C([0,T];H^{s-1}(\T))$.
\begin{proposition}
\label{prop:diffHs-1}

Let $s> \frac 52$.
Assume that $F$ satisfies \eqref{Fbez}.
Then,
there exists a constant $C_3( \| \phi \|_{H^s})>0$ such that
\[
\| u^{\eps_1} - u^{\eps_2} \|_{L_T^\infty H^{s-1}}
\le C_3(\| \phi \|_{H^s}) |\eps_1 - \eps_2|^{\frac 12}
\]
for $\eps_1, \eps_2 \in (0,1)$,
where $T$ is given in \eqref{time1}.
\end{proposition}

\begin{proof}
Let $\eps_1 ,\eps_2 \in (0,1)$.
Without loss of generality,
we may assume that $1>\eps_1>\eps_2>0$.
We set
\[
\breve u = u^{\eps_1} - u^{\eps_2},
\qquad
\breve v = v^{\eps_1} - v^{\eps_2}
\]
for short.

It follows from \eqref{ueps} that
\begin{align}
\dt \breve u + i \dx^2 \breve u
&=
\eps _1 \dx^2 \breve u
+ (\eps_1 - \eps_2) \dx^2 u^{\eps_2}
+
F^{\eps_1} - F^{\eps_2} .
\label{NLSdu}
\end{align}
Moreover,
by \eqref{veps}, we have
\begin{equation}
\begin{aligned}
\dt \breve v + i\dx^2 \breve v
&=
\eps _1 \dx^2 \breve v
+ (\eps_1 - \eps_2) \dx^2 v^{\eps_2}
+
F_{\be}^{\eps_1} \dx \breve v
+
F_{\cj \be}^{\eps_1} \cj{\dx \breve v}
\\
&\quad
+ (F_{\be}^{\eps_1} - F_{\be}^{\eps_2}) \dx v^{\eps_2}
+ (F_{\cj \be}^{\eps_1} - F_{\cj \be}^{\eps_2}) \cj {\dx v^{\eps_2}}
+ \breve R
,
\end{aligned}
\label{NLSdv}
\end{equation}
where
\begin{align*}
\breve R
&:=
F_{\al}^{\eps_1} \breve v
+ F_{\cj \al}^{\eps_1} \cj {\breve v}
+ (F_{\al}^{\eps_1} - F_{\al}^{\eps_2}) v^{\eps_2}
+ ( F_{\cj \al}^{\eps_1} - F_{\cj \al}^{\eps_2}) \cj {v^{\eps_2}}.
\end{align*}
We apply a gauge transformation to eliminate the problematic term in the third part on the right-hand side of \eqref{NLSdv}.
Define
\[
\breve V
=
e^{- \Ld^{\eps_1}} \breve v,
\]
where
$\Ld^{\eps_1}$ is defined in \eqref{Ldep}.
Then,
by the same calculation as in \eqref{NLS2eps},
$\breve V$ satisfies
\begin{equation}
\begin{aligned}
\dt \breve V + i \dx^2 \breve V
&=
\eps _1 \dx^2 \breve V
+ (\eps_1 - \eps_2) 
e^{- \Ld^{\eps_1}} \dx^2 v^{\eps_2}
\\
&\quad
+
(P_0 F_\be^{\eps_1}) \dx \breve V
- 
i \eps_1 (P_{\neq 0} F_{\be}^{\eps_1}) \dx \breve V
\\
&\quad
+
F_{\cj \be}^{\eps_1} e^{-\Ld^{\eps_1} + \cj{\Ld^{\eps_1}}} \cj {\dx \breve V}
\\
&\quad
+
e^{-\Ld^{\eps_1}}(F_{\be}^{\eps_1} - F_{\be}^{\eps_2}) \dx v^{\eps_2}
+
e^{-\Ld^{\eps_1}} (F_{\cj \be}^{\eps_1} - F_{\cj \be}^{\eps_2}) \cj{\dx v^{\eps_2}}
\\
&\quad
+
\breve \fR,
\label{NLSdW0}
\end{aligned}
\end{equation}
where 
\begin{equation}
\begin{aligned}
\breve \fR 
&:=
(P_0 F_\be^{\eps_1}) \dx \Ld^{\eps_1} \cdot \breve V
- i \eps_1 (P_{\neq 0} F_\be^{\eps_1}) \dx \Ld^{\eps_1} \cdot \breve V
\\
&\quad
+
F_{\cj \be}^{\eps_1} e^{-\Ld^{\eps_1} + \cj{\Ld^{\eps_1}}} \cj{\dx \Ld^{\eps_1} \cdot \breve V}
+ e^{- \Ld^{\eps_1}} \breve R
\\
&\quad
+
\big(
-\dt \Ld^{\eps_1}
+ (i-\eps_1) (\dx \Ld^{\eps_1})^2
- (i-\eps_1) \dx^2 \Ld^{\eps_1} \big)
\breve V
.
\label{wtfR00}
\end{aligned}
\end{equation}

Set
\begin{equation}
{\breve E}_s (t)
:=
\sqrt{
\|\breve u(t)\|_{H^{s-1}}^2
+
\|\breve V(t)\|_{H^{s-1}}^2
}.
\label{Ewt0}
\end{equation}
We prove that
there exists a constant $\breve C_3( \| \phi \|_{H^s} )>0$ such that
\begin{equation}
\dt \big( \breve E_{s-1}(t)^2 \big)
\le \breve C_3( \| \phi \|_{H^s})
\big(
|\eps_1-\eps_2|
+
\breve E_{s-1}(t)^2 \big)
\label{diffEwt}
\end{equation}
for $t \in (0,T)$ and $1>\eps_1> \eps_2>0$.
Then,
by solving this differential inequality and $\breve E_{s-1}(0)=0$,
we have
\[
\sup_{t \in [0,T]}
\breve E_{s-1}(t)^2
\le
|\eps_1-\eps_2|
e^{\breve C_3 (\| \phi \|_{H^s}) T}
.
\]
Propositions \ref{prop:bili} and \ref{prop:dtEne} and Lemma \ref{lem:Enebd1} yield that
\begin{equation}
\begin{aligned}
\| \breve v \|_{L_T^\infty H^{s-2}}
&=
\big\| e^{\Ld^{\eps_1}} \breve V \big\|_{L_T^\infty H^{s-2}}
\les
\big\| e^{\Ld^{\eps_1}} \big\|_{L_T^\infty H^{s-2}} \| \breve V \|_{L_T^\infty H^{s-2}}
\\
&\le C( \|\phi \|_{H^s})
\sup_{t \in [0,T]} \breve E_{s-1} (t).
\end{aligned}
\label{diffbd1}
\end{equation}
Hence, we obtain the desired bound.

As in the proof of Proposition \ref{prop:dtEne},
we suppress the time dependence below.
It follows from \eqref{NLSdu} and $1>\eps_1>\eps_2>0$ that
\begin{align*}
&\frac 12 \dt \| \breve u \|_{H^{s-2}}^2
\\
&=
- \eps_1 \| \dx \breve u \|_{H^{s-2}}^2
- (\eps_1-\eps_2) \Re (\dx u^{\eps_2}, \dx \breve u )_{H^{s-2}}
\\
&\quad
+ \Re (F^{\eps_1} - F^{\eps_2}, \breve u)_{H^{s-2}}
\\
&\le
|\eps_1-\eps_2|
\cdot
\| u^{\eps_2} \|_{H^{s-1}}
\| \breve u \|_{H^{s-1}}
+ \| F^{\eps_1} - F^{\eps_2} \|_{H^{s-2}}
\| \breve u \|_{H^{s-2}}.
\end{align*}
From Propositions \ref{prop:bili} and \ref{prop:dtEne},
and \eqref{diffbd1},
we have
\begin{align*}
\| F^{\eps_1} - F^{\eps_2} \|_{H^{s-2}}
&\le
C( \| \phi \|_{H^s})
(\| \breve u \|_{H^{s-2}} + \| \breve v \|_{H^{s-2}} )
\le 
C( \| \phi \|_{H^s})
\breve E_{s-1}.
\end{align*}
Hence,
we obtain that
\begin{equation}
\dt \| \breve u \|_{H^{s-2}}^2
\le
C( \| \phi \|_{H^s})
\big(
|\eps_1-\eps_2|
+
\breve E_{s-1} \|\breve u\|_{H^{s-2}}
\big).
\label{est3u}
\end{equation}

It follows from \eqref{NLSdW0} that
\begin{equation}
\begin{aligned}
\frac 12 \dt \| \breve V \|_{H^{s-2}}^2
&=
- \eps_1 \| \dx \breve V \|_{H^{s-2}}^2
\\
&\quad
- (\eps_1-\eps_2) \Re (\dx v^{\eps_2}, \dx \breve V )_{H^{s-2}}
\\
&\quad
+
\Re \big( (P_0 F_{\be}^{\eps_1}) \dx \breve V , \breve V \big)_{H^{s-2}}
\\
&\quad
+
\Re \big( -i \eps_1  (P_{\neq 0} F_{\be}^{\eps_1}) \dx \breve V , \breve V \big)_{H^{s-2}}
\\
&\quad
+ 
\Re \big( F_{\cj \be}^{\eps_1} e^{-\Ld^{\eps_1} + \cj{\Ld^{\eps_1}}} \cj {\dx \breve V}, \breve V \big)_{H^{s-2}}
\\
&\quad
+
\Re \big( e^{-\Ld^{\eps_1}}(F_{\be}^{\eps_1}- F_{\be}^{\eps_2}) \dx v^{\eps_2}, \breve V \big)_{H^{s-2}}
\\
&\quad
+
\Re \big(e^{-\Ld^{\eps_1}} (F_{\cj \be}^{\eps_1}- F_{\cj \be}^{\eps_2}) \cj{\dx v^{\eps_2}} , \breve V \big)_{H^{s-2}}
\\
&\quad
+
\Re (\breve \fR ,\breve V)_{H^{s-2}}
\\
&=:
- \eps_1 \| \dx \breve V \|_{H^{s-2}}^2
+
\sum_{j=1}^7 \I_j .
\end{aligned}
\label{enwtW0}
\end{equation}
From Proposition \ref{prop:dtEne},
we have
\begin{equation}
|\I_1|
\le
C( \| \phi \|_{H^s})
|\eps_1-\eps_2|.
\label{est3-00}
\end{equation}
It follows from
the assumption \eqref{Fbez}
and the same calculation as in \eqref{est1a}
that
\begin{equation}
\I_2
=
0.
\label{est3-1}
\end{equation}

For $\I_3$,
we apply a similar calculation as in \eqref{est2a}.
Namely,
we decompose $\I_3$ as follows:
\begin{equation}
\begin{aligned}
\I_3
&=
\eps_1
\Im  \big( \big[ \jb{\dx}^{s-2},P_{\neq 0} F^{\eps_1}_\be \big] \dx \breve V, \jb{\dx}^{s-2} \breve V \big)_{L^2}
\\
&\quad
+
\eps_1
\Im \big( (P_{\neq 0} F^{\eps_1}_\be) \jb{\dx}^{s-2} \dx \breve V, \jb{\dx}^{s-2} \breve V \big)_{L^2}
\\
&=: \I_{3,1} + \I_{3,2} .
\label{est2aw}
\end{aligned}
\end{equation}
It follows from
Propositions \ref{prop:Com1} and \ref{prop:dtEne}
that
\begin{equation}
\begin{aligned}
|\I_{3,1}|
&\le
\big\| \big[ \jb{\dx}^{s-2},P_{\neq 0} F^{\eps_1}_\be \big] \dx \breve V \big\|_{L^2}
\| \breve V \|_{H^{s-2}}
\\
&\les
\|P_{\neq 0} F^{\eps_1}_\be\|_{H^{s-1}} \|\breve V \|_{H^{s-2}}
\|\breve V \|_{H^{s-2}}
\\
&\le
C( \| \phi \|_{H^s} ) \breve E_{s-1} \| \breve V \|_{H^{s -2}}.
\end{aligned}
\label{est2bw}
\end{equation}
A direct calculation yields that
\begin{equation}
\begin{aligned}
|\I_{3,2}|
&\le
\eps_1
\| P_{\neq 0} F^{\eps_1}_\be \|_{L^\infty}
\| \jb{\dx}^{s-2} \dx \breve V \|_{L^2} \| \jb{\dx}^{s-2} \breve V \|_{L^2}
\\
&\le
\frac{\eps_1}2 \| \dx \breve V \|_{H^{s-2}}^2
+
\frac{\eps_1}2
\| P_{\neq 0} F^{\eps_1}_\be \|_{L^\infty}^2
\| \breve V \|_{H^{s-2}}^2
\\
&\le
\frac{\eps_1}2 \| \dx \breve V \|_{H^{s-2}}^2
+
C( \| \phi \|_{H^s}) \breve E_{s-1} \| \breve V \|_{H^{s-2}}.
\end{aligned}
\label{est2caw}
\end{equation}
From \eqref{est2aw}--\eqref{est2caw},
we obtain that
\begin{equation}
|\I_3|
\le
\frac{\eps_1}2 \| \dx \breve V \|_{H^{s-2}}^2
+
C(\| \phi \|_{H^s})
\breve E_{s-1} \| \breve V \|_{H^{s-2}}.
\label{est3-2}
\end{equation}
Moreover, the same calculation as in \eqref{est31} yields that
\begin{equation}
|\I_4|
\le
C(\| \phi \|_{H^s})
\breve E_{s-1} \| \breve V \|_{H^{s-2}}.
\label{est3-3}
\end{equation}

Propositions \ref{prop:bili} and \ref{prop:dtEne}, Lemma \ref{lem:Enebd1},
and \eqref{diffbd1}
yield that
we have
\begin{equation}
\begin{aligned}
|\I_5|
&\le
\big\| e^{-\Ld^{\eps_1}} (F^{\eps_1}_\be - F^{\eps_2}_\be) \dx v^{\eps_2} \big\|_{H^{s-2}}
\| \breve V \|_{H^{s-2}}
\\
&\le
\big\| e^{-\Ld^{\eps_1}} \big\|_{H^{s-2}} \| F^{\eps_1}_\be - F^{\eps_2}_\be \|_{H^{s-2}}
\| \dx v^{\eps_2} \|_{H^{s-2}}
\| \breve V \|_{H^{s-2}}
\\
&\le
C(\| \phi \|_{H^s})
\breve E_{s-1} \| \breve V \|_{H^{s-2}}.
\end{aligned}
\label{est3-4}
\end{equation}
The same calculation as in \eqref{est3-4}, we obtain
\begin{equation}
|\I_6|
\le
C(\| \phi \|_{H^s}) \breve E_{s-1} \|\breve V \|_{H^{s-2}}.
\label{est3-5}
\end{equation}
By \eqref{wtfR00},
$\breve \fR$ contains at least one of $\breve u$ and $\breve V$.
Hence, we have
\begin{equation}
|\I_7|
\les
\| \breve \fR \|_{H^{s-2}} \|\breve V \|_{H^{s-2}}
\le
C(\| \phi \|_{H^s})
\breve E_{s-1}
\|\breve V \|_{H^{s-2}}.
\label{est3-7}
\end{equation}

It follows from \eqref{enwtW0}--\eqref{est3-1} and \eqref{est3-2}--\eqref{est3-7} that
\begin{equation}
\dt \| \breve V \|_{H^{s-2}}^2
\le
C( \| \phi \|_{H^s})
\big(
|\eps_1-\eps_2|
+
\breve E_{s-1} \|\breve V \|_{H^{s-2}}
\big).
\label{est3w}
\end{equation}
By \eqref{Ewt0}, \eqref{est3u}, and \eqref{est3w},
we obtain \eqref{diffEwt}.
\end{proof}

Proposition \ref{prop:diffHs-1} yields that
$\{u^{\eps}\}$ converges in $C([0,T]; H^{s-1}(\T))$.
Hence,
we have
\[
u = \lim_{\eps \to +0} u^\eps \in C([0,T]; H^{s-1}(\T)),
\]
where $u$ is the solution to \eqref{NLS} constructed before Proposition \ref{prop:diffHs-1}.

\begin{remark}
\label{rem:pers}
\rm
If $\phi \in H^r(\T)$ for some $r>s$,
the solution $u$ to \eqref{NLS} constructed above satisfies
\[
u \in L^\infty([0,T]; H^r(\T)) \cap C([0,T]; H^{r-1}(\T)).
\]
Moreover,
we have the bound
\begin{equation}
\| u \|_{L_T^\infty H^r} \le C_2 ( \| \phi \|_{H^s}) \| \phi \|_{H^r}.
\label{est:solbd1}
\end{equation}
Note that $r=s+2$ suffices for the argument in Subsection \ref{Subsec:code} below.
\end{remark}

A slight modification of the proof of Proposition \ref{prop:diffHs-1} yields the following:

\begin{corollary}
\label{cor:diffHs-1s}
Let $s> \frac 52$.
Assume that $F$ satisfies \eqref{Fbez}.
Let $u^{(j)} \in C([0,T]; H^s(\T))$ be a solution to \eqref{NLS} in $H^s(\T)$ on $[0,T]$ with the initial data $\phi^{(j)} \in H^s(\T)$ for $j=1,2$.
Then,
there exists a constant 
\[
C_4( \| u^{(1)} \|_{L_T^\infty H^s}, \| u^{(2)} \|_{L_T^\infty H^s} )>0
\]
such that
\[
\| u^{(1)} - u^{(2)} \|_{L_T^\infty H^{s-1}}
\le C_4( \| u^{(1)} \|_{L_T^\infty H^s}, \| u^{(2)} \|_{L_T^\infty H^s} ) \| \phi^{(1)} - \phi^{(2)} \|_{H^{s-1}}.
\]
\end{corollary}

\begin{proof}
It follows from $u^{(j)} \in C([0,T]; H^s(\T))$ and Remark \ref{rem:sold} that
\[
v^{(j)} := \dx u^{(j)} \in C ([0,T]; H^{s-1}(\T)) \cap C^1([0,T]; H^{s-3}(\T)).
\]
Accordingly,
we have
\begin{align*}
\dt \| v^{(j)} (t) \|_{H^{s-2}}^2
&= 2 \Re ( \jb{\dx}^{s-3} \dt v^{(j)} (t), \jb{\dx}^{s-1} v^{(j)} (t))_{L^2}
\end{align*}
for $t \in [0,T]$.
Therefore,
we can apply the same argument as in the proof of Proposition \ref{prop:diffHs-1}
replacing $u^{\eps_1}$ and $u^{\eps_2}$ with $u^{(1)}$ and $u^{(2)}$, respectively.
\end{proof}

In particular,
Corollary \ref{cor:diffHs-1s} shows the uniqueness of a solution to \eqref{NLS}.

\subsection{Continuous dependence on initial data}
\label{Subsec:code}

To prove well-posedness,
we use the Bona-Smith-type approximation as in \cite{BoSm75}.
Let $s> \frac 52$ and $\phi \in H^s(\T)$.
For $N \in \N$, we set 
\begin{equation}
\phi_N (x) := \sum_{\substack{k \in \Z \\ |k| \le N}} \ft \phi(k) e^{ikx}.
\label{BS_ini1}
\end{equation}
A simple calculation yields the following:

\begin{lemma}
\label{lem:BoSm}
The truncated function $\phi_N$ defined in \eqref{BS_ini1} satisfies the followings:
\begin{enumerate}
\item
$\phi_N \in C^{\infty}(\T)$ for any $N \in \N$.

\item
$\lim_{N \to \infty} \|\phi_N - \phi\|_{H^s} =0$.

\item
For $r>s$, we have $\| \phi_N \|_{H^r} \les N^{r-s} \| \phi \|_{H^s}$.

\item
For $r<s$, we have
$\|\phi_N - \phi \|_{H^r} \les N^{-(s-r)} \|\phi\|_{H^s}$.
\end{enumerate}
\end{lemma} 

Let $N \in \N$ and
let $u_N$ be the the solution to \eqref{NLS} with the initial data
$\phi_N$.
Since $\phi_N \in C^\infty(\T)$,
Remark \ref{rem:pers} yields that
(at least)
\[
u_N \in C([0,T]; H^{s+1}(\T)) \cap C^1([0,T]; H^{s-1}(\T)).
\]
Set
\[
v_N = \dx u_N, \quad w_N = \dx v_N.
\]
For simplicity, we write
\begin{align*}
F^{(N)}
&= F(u_N, v_N, \cj {u_N}, \cj {v_N}),
&
F_{\al}^{(N)}
&= F_\al (u_N, v_N, \cj {u_N}, \cj {v_N}),
\quad
\cdots
\end{align*}

Similarly as in the proof of Proposition \ref{prop:diffHs-1},
for $N,M \in \N$ with $N>M$,
we define
\[
\wt u := u_N - u_M,
\qquad
\wt v := \dx \wt u,
\qquad
\wt w := \dx^2 \wt u
\]
for short.
By the same calculation as in \eqref{ueps}, \eqref{veps}, and \eqref{weps},
$\wt u$, $\wt v$, and $\wt w$ satisfy the followings:
\begin{align*}
\dt \wt u + i \dx^2 \wt u
&= F^{(N)} - F^{(M)} ,
\\
\dt \wt v + i \dx^2 \wt v
&=
R_1^{(N)} - R_1^{(M)},
\\
\dt \wt w + i \dx^2 \wt w
&=
F_{\be}^{(N)} \dx \wt w
+
F_{\cj \be}^{(N)} \cj{\dx \wt w}
\notag
\\
&\quad
+ (F_{\be}^{(N)} - F_{\be}^{(M)}) \dx w_M
+
(F_{\cj \be}^{(N)} - F_{\cj \be}^{(M)})\cj{\dx w_M}
\\
&\quad
+
R_2^{(N)} - R_2^{(M)},
\notag
\end{align*} 
where 
$R_j^{(\l)}$ is
harmless
for $\l \in \{N,M\}$ and $j \in \{1,2 \}$.
We apply the gauge transformation to eliminate the problematic term in the first part on the right-hand side of the equation for $\wt w$.
Define
\begin{align*}
\Ld^{(N)}
&:=
-\frac i2 \dx^{-1} F_{\be}^{(N)} \wt w,
&
\wt W
&:=
e^{- \Ld^{(N)}} \wt w.
\end{align*}
By the same calculation as in \eqref{NLS2eps},
$\wt W$ satisfies
\begin{equation}
\begin{aligned}
\dt \wt W + i \dx^2 \wt W
&=
(P_0 F_{\be}^{(N)}) \dx \wt W  
+ 
F_{\cj \be}^{(N)} e^{-\Ld^{(N)} + \cj{\Ld^{(N)}}} \cj {\dx \wt W}
\\
&\quad
+
e^{-\Ld^{(N)}}(F_{\be}^{(N)} - F_{\be}^{(M)}) \dx w_M
\\
&\quad
+
e^{-\Ld^{(N)}} (F_{\cj \be}^{(N)} - F_{\cj \be}^{(M)}) \cj{\dx w_M}
+
\wt \fR,
\label{NLSdW}
\end{aligned}
\end{equation}
where 
\begin{equation}
\begin{aligned}
\wt \fR 
&:=
(P_0 F_{\be}^{(N)}) \dx \Ld^{(N)} \cdot \wt W  
+
F_{\cj \be}^{(N)} e^{-\Ld^{(N)} + \cj{\Ld^{(N)}}} \cj {\dx \Ld^{(N)}} \cdot \cj{\wt W}
\\
&\quad
+
e^{-\Ld^{(N)}} ( R_2^{(N)} - R_2^{(M)})
\\
&\quad
+
\big( -\dt \Ld^{(N)} + i (\dx \Ld^{(N)})^2 -i \dx^2 \Ld^{(N)} \big) \wt W.
\label{wtfR} 
\end{aligned}
\end{equation}

Set
\begin{align}
{\wt E}_s(t)
&:=
\sqrt{
\|\wt u(t)\|_{H^{s-2}}^2
+
\|\wt v(t)\|_{H^{s-2}}^2
+
\|\wt W(t)\|_{H^{s-2}}^2}.
\label{Ewt}
\end{align}
The following estimate plays a key role.

\begin{proposition}
\label{prop:diffNM}
Let $\phi \in H^s(\T)$ and $s_0 \in (\frac 52, s)$.
Assume that $F$ satisfies \eqref{Fbez}.
Then, there exists a constant $C( \| \phi \|_{H^s})$
such that
\[
\dt \wt E_s (t)
\le
C( \| \phi \|_{H^s})
\big(
\wt E_s (t)
+
M^{- (s-s_0)} \big)
\]
for $t \in (0,T)$ and $N,M \in \N$ with $N>M$.
\end{proposition}

\begin{proof}
For simplicity, we suppress the time dependence in this proof.
%
From Proposition \ref{prop:bili} and \eqref{est:solbd1},
we have 
\begin{align}
\dt \| \wt u \|_{H^{s-2}}^2
&\le
2
\|F^{(N)} - F^{(M)}\|_{H^{s-2}}\|\wt u\|_{H^{s-2}}
\notag
\\
&\le
C(\| \phi \|_{H^s})
\wt E_s \cdot \|\wt u\|_{H^{s-2}},
\label{est4u}
\\
\dt \| \wt v \|_{H^{s-2}}^2
&\le
2 \| R_1^{(N)} - R_2^{(M)} \|_{H^{s-2}}
\|\wt v\|_{H^{s-2}}
\notag
\\
&\le
C(\| \phi \|_{H^s})
\wt E_s \cdot 
\|\wt v\|_{H^{s-2}}.
\label{est4v}
\end{align}
It follows from \eqref{NLSdW} that
\begin{equation}
\begin{aligned}
\frac 12 \dt \| \wt W \|_{H^{s-2}}^2
&= \Re ( \jb{\dx}^{s-3} \dt \wt W, \jb{\dx}^{s-1} \wt W)_{L^2}
\\
&=
\Re \big( (P_0 F_{\be}^{(N)}) \dx \wt W , \wt W \big)_{H^{s-2}}
\\
&\quad
+ 
\Re \big( F_{\cj \be}^{(N)} e^{-\Ld^{(N)} + \cj{\Ld^{(N)}}} \cj {\dx \wt W}, \wt W \big)_{H^{s-2}}
\\
&\quad
+
\Re \big( e^{-\Ld^{(N)}}(F_{\be}^{(N)} - F_{\be}^{(M)}) \dx w_M , \wt W \big)_{H^{s-2}}
\\
&\quad
+
\Re \big(e^{-\Ld^{(N)}} (F_{\cj \be}^{(N)} - F_{\cj \be}^{(M)}) \cj{\dx w_M} , \wt W \big)_{H^{s-2}}
\\
&\quad
+
\Re (\wt \fR ,\wt W)_{H^{s-2}}
\\
&=:\sum_{j=1}^5 \I_j .
\end{aligned}
\label{enwtW}
\end{equation}

From the assumption \eqref{Fbez},
the same calculation as in \eqref{est1a} shows that
\begin{equation}
\I_1
=
\Re P_0 F_\be^{(N)}
\cdot \frac 12
\int_\T \dx \big( | \jb{\dx}^{s-2} \wt W|^2 \big) dx
=
0.
\label{est4-1}
\end{equation}
The same calculation as in \eqref{est31} (or \eqref{est3-3}) yields that
\begin{equation}
|\I_2|
\le
C(\| \phi \|_{H^s})
\wt E_s \| \wt W\|_{H^{s-2}}.
\label{est4-2}
\end{equation}

We write $\I_3$ as follows:
\begin{equation}
\begin{aligned}
\I_3
&=
\Re \big( \big[ \jb{\dx}^{s-2}, e^{-\Ld^{(N)}}(F_{\be}^{(N)} - F_{\be}^{(M)}) \big] \dx w_M, \jb{\dx}^{s-2} \wt W \big)_{L^2}
\\
&\quad
+
\Re \big( e^{-\Ld^{(N)}}(F_{\be}^{(N)} - F_{\be}^{(M)}) \jb{\dx}^{s-2} \dx w_M, \jb{\dx}^{s-2} \wt W \big)_{L^2}
\\
&=: \I_{3,1} + \I_{3,2} .
\label{est3a'}
\end{aligned}
\end{equation}
By Propositions \ref{prop:bili2} and \ref{prop:Com1},
we obtain
\begin{equation}
\begin{aligned}
|\I_{3,1}| 
&\les
\big\| \big[ \jb{\dx}^{s-2},e^{-\Ld^{(N)}}(F_{\be}^{(N)} - F_{\be}^{(M)}) \big] \dx w_M \big\|_{L^2} \| \jb{\dx}^{s-2} \wt W \|_{L^2}
\\
&\les
\big\| e^{-\Ld^{(N)}}(F_{\be}^{(N)} - F_{\be}^{(M)}) \big\|_{H^{s-1}} \|w_M\|_{H^{s-2}} \|\wt W\|_{H^{s-2}}
\\
&\le
C(\| \phi \|_{H^s})
\wt E_s \|\wt W\|_{H^{s-2}}.
\label{est4-31}
\end{aligned}
\end{equation}
H\"older' inequality yields that
\begin{equation}
\begin{aligned}
|\I_{3,2} |
&\les
\big\| e^{-\Ld^{(N)}}(F_{\be}^{(N)} - F_{\be}^{(M)}) \big\|_{L^\infty}
\| w_M \|_{H^{s-1}} \|\wt W\|_{H^{s-2}}.
\end{aligned}
\label{est3b'}
\end{equation}
It follows from $s_0 \in (\frac 52, s)$ and Corollary \ref{cor:diffHs-1s} that
\begin{equation}
\begin{aligned}
\big\| e^{-\Ld^{(N)}}(F_{\be}^{(N)} - F_{\be}^{(M)}) \big\|_{L^\infty}
&
\les
\big\| e^{-\Ld^{(N)}} \big\|_{L^\infty} \| F_{\be}^{(N)} - F_{\be}^{(M)} \|_{L^\infty}
\\
&\le
C(\| \phi \|_{H^s})
(\| \wt u \|_{L^\infty} + \| \dx \wt u \|_{L^\infty})
\\
&\le
C(\| \phi \|_{H^s})
\| \wt u \|_{H^{s_0-1}}
\\
&\le
C(\| \phi \|_{H^s})
\| \phi_N - \phi_M \|_{H^{s_0-1}}.
\end{aligned}
\label{est3c'}
\end{equation}
From Remark \ref{rem:pers},
we have
\begin{equation}
\|w_M\|_{H^{s-1}}
\le
\| u_M \|_{H^{s+1}}
\le
C( \| \phi \|_{H^s})
(1+\| \phi_M \|_{H^{s+1}})
.
\label{est3h'}
\end{equation}
Combining \eqref{est3b'}--\eqref{est3h'},
we obtain that
\[
|\I_{3,2} |
\le
C(\| \phi \|_{H^s})
\|\phi_N - \phi_M\|_{H^{s_0-1}} (1+\|\phi_M\|_{H^{s+1}})
\|\wt W\|_{H^{s-2}}.
\]
Lemma \ref{lem:BoSm} with $N>M$ yields that
\begin{equation}
\begin{aligned}
|\I_{3,2} |
&\le
C(\| \phi \|_{H^s})
M^{-(s-s_0)}
\|\wt W\|_{H^{s-2}}.
\label{est4-32}
\end{aligned}
\end{equation}
With \eqref{est3a'}, \eqref{est4-31}, and \eqref{est4-32},
we obtain that
\begin{equation}
|\I_3|
\le
C(\| \phi \|_{H^s})
\big( \wt E_s + M^{-(s-s_0)} \big)
\|\wt W\|_{H^{s-2}}.
\label{est4-32bb}
\end{equation}

The same calculation as in \eqref{est4-32bb} yields that
\begin{equation}
\begin{aligned}
|\I_4|
\le
C(\| \phi \|_{H^s})
\big(\wt E_s + M^{-(s-s_0)}\big)
\|\wt W\|_{H^{s-2}}
.
\end{aligned}
\label{est4-4}
\end{equation}
By \eqref{wtfR},
$\wt \fR$ contains at least one of $\wt u, \wt v$, and $\wt W$.
Hence,
we have 
\begin{equation}
|\I_5|
\les
\| \wt \fR \|_{H^{s-2}} \|\wt W\|_{H^{s-2}}
\le
C(\| \phi \|_{H^s})
\wt E_s
\|\wt W\|_{H^{s-2}}.
\label{est4-5}
\end{equation}

It follows from \eqref{enwtW}--\eqref{est4-2} and \eqref{est4-32bb}--\eqref{est4-5} that
\begin{equation}
\dt \| \wt W \|_{H^{s-2}}^2
\le
C(\| \phi \|_{H^s})
\big(\wt E_s
+ M^{-(s-s_0)} \big) \|\wt W\|_{H^{s-2}}.
\label{estwtW}
\end{equation}
From \eqref{Ewt}--\eqref{est4v} and \eqref{estwtW},
we obtain desired bound.
\end{proof}

Proposition \ref{prop:diffNM} yields that
\[
\wt E_s(t)
\le
\big(\wt E_s(0) +M^{- (s-s_0)}\big) e^{C(\| \phi \|_{H^s}) t}.
\]
By a similar calculation as in \eqref{diffbd1},
we have
\begin{equation}
\| u_N - u_M \|_{L_T^\infty H^s}
\le
C(\| \phi \|_{H^s})
\big( \| \phi_N - \phi_M \|_{H^s} + M^{-(s-s_0)} \big)
\label{BSdiffa}
\end{equation}
for $N>M$.
Hence,
$\{ u_N \}_{N \in \N}$ converges to $u$ constructed at the end of Subsection \ref{SUBSEC:energy} in $C([0,T]; H^s(\T))$.
In particular, we have $u \in C([0,T]; H^s(\T))$.

Finally, we prove the well-posedness result in Theorem \ref{thm:equiv}.

\begin{proof}[Proof of the well-posedness for \eqref{NLS}]
Let $\{ \phi^{(\l)} \}_{\l \in \N} \subset H^s(\T)$ converge to $\phi$ in $H^s(\T)$.
Moreover, let $\phi^{(\l)}_N$ be the truncation of $\phi^{(\l)}$ defined in \eqref{BS_ini1}.
Set $u^{(\l)}$ and $u^{(\l)}_N$ as the solutions to \eqref{NLS} with the initial data  $\phi^{(\l)}$ and $\phi^{(\l)}_N$, respectively.

The triangle inequality and \eqref{BSdiffa} yield that
\begin{align*}
\|u^{(\l)} - u\|_{L_T^\infty H^s}
&\le
\|u^{(\l)} - u^{(\l)}_N \|_{L_T^\infty H^s}
+ \|u^{(\l)}_N - u _N \|_{L_T^\infty H^s}
\\
&\quad
+ \|u_N - u\|_{L_T^\infty H^s}
\\
&\le
C(\| \phi \|_{H^s})
\big(
\| \phi^{(\l)} - \phi^{(\l)}_N \|_{H^s}
+ \| \phi_N - \phi \|_{H^s}
+ N^{-(s-s_0)} \big)
\\
&\quad
+ \|u^{(\l)}_N - u _N \|_{L_T^\infty H^s}.
\end{align*}
Since $\phi_N$ and $\phi_N^{(\l)}$ are the truncation of $\phi$ and $\phi^{(\l)}$, respectively,
we have
\[
\| \phi^{(\l)} - \phi^{(\l)}_N \|_{H^s}
\le
2 \| \phi^{(\l)} - \phi \|_{H^s}
+
\| \phi_N - \phi \|_{H^s}.
\]
Accordingly, it holds that
\begin{equation}
\begin{aligned}
&\|u^{(\l)} - u\|_{L_T^\infty H^s}
\\
&\le
C(\| \phi \|_{H^s})
\big(
\| \phi^{(\l)} - \phi \|_{H^s}
+ \| \phi_N - \phi \|_{H^s}
+ N^{-(s-s_0)} \big)
\\
&\quad
+ \|u^{(\l)}_N - u _N \|_{L_T^\infty H^s}.
\end{aligned}
\label{convNN0}
\end{equation}
Once we have
\begin{equation}
\lim_{\l \to \infty} u^{(\l)}_N = u_N
\label{convNN}
\end{equation}
in $C([0,T]; H^s(\T))$ for $N \in \N$,
the desired convergence
\[
\lim_{\l \to \infty} u^{(\l)} = u
\]
in $C([0,T];H^s(\T))$
follows from \eqref{convNN0} and \eqref{convNN}.

We prove \eqref{convNN}.
From Remark \ref{rem:pers}, we have
\begin{align*}
&\| u_N^{(\l)} \|_{L_T^\infty H^{s+1}}
+ \| u_N \|_{L_T^\infty H^{s+1}}
\\
&\le C( \|\phi_N \|_{H^s}) (1+ \| \phi_N^{(\l)} \|_{H^{s+1}}+ \| \phi_N \|_{H^{s+1}})
\\
&\le C( \|\phi \|_{H^s}) N.
\end{align*}
Hence, the same calculation as in the proof of Proposition \ref{prop:diffNM} implies that
\begin{align*}
&\| u_N^{(\l)} - u_N  \|_{L_T^\infty H^s}
\\
&\le
C( \|\phi \|_{H^s})
\Big(
\| \phi_N^{(\l)} - \phi_N  \|_{H^s}
+ N \| \phi_N^{(\l)} - \phi_N  \|_{H^{s_0-1}}
\Big)
e^{C( \|\phi \|_{H^s}) T},
\end{align*}
which yields \eqref{convNN}.
This concludes the proof.
\end{proof}

\section{Non-existence}
\label{Sec:NE}

In this section, we prove that \eqref{NLS} has no solution if \eqref{Fbez2} holds.
The main result in this section is the following.

\begin{theorem}
\label{thm:NE2}
Let $s> \frac 52$ and $T>0$.
Assume that
$u \in C([0,T]; H^s(\T))$ is a solution to \eqref{NLS}.

\begin{enumerate}
\item
If
\[
\int_\T \Im F_\be ( \phi,\dx \phi,\cj{\phi},\cj{\dx \phi} ) dx >0,
\]
then we have
$P_- \phi \in H^{s+\dl}(\T)$
for any  $\dl \in (0,\frac 12)$.

\item
If
\[
\int_\T \Im F_\be ( \phi,\dx \phi,\cj{\phi},\cj{\dx \phi} ) dx <0,
\]
then we have
$P_+ \phi \in H^{s+\dl}(\T)$
for any  $\dl \in (0,\frac 12)$.
\end{enumerate}

The same conclusion holds for a solution $u \in C([-T,0]; H^s(\T))$ by interchanging $P_- \phi$ and $P_+ \phi$.
\end{theorem}

Once Theorem \ref{thm:NE2} is obtained,
Theorem \ref{thm:nonexi}
follows from choosing $\phi \in H^s(\T)$ with $P_\pm \phi \notin H^{s+\dl}(\T)$ and
\[
\int_\T \Im F_\be ( \phi,\dx \phi,\cj{\phi},\cj{\dx \phi} ) dx \neq 0.
\]
By using an approximation argument, we can choose such initial data when \eqref{Fbez2} holds.

\begin{lemma}
\label{lem:NIini1}
Let $s>\frac 52$ and $\psi \in H^s(\T)$ satisfy \eqref{Fbez2}.
Then,
for $\dl \in (0,\frac 12)$,
there exists $\phi \in H^s(\T)$ that satisfies the followings:
\begin{enumerate}
\item
$P_+ \phi \notin H^{s+\dl}(\T)$ and 
$P_- \phi \notin H^{s+\dl}(\T)$;

\item
$\int_\T \Im F_\be ( \phi,\dx \phi,\cj{\phi},\cj{\dx \phi} ) dx$
and $\int_\T \Im F_\be ( \psi,\dx \psi,\cj{\psi},\cj{\dx \psi} ) dx$
have the same sign.
\end{enumerate}
\end{lemma}

\begin{proof}
We consider the two cases:
\begin{itemize}
\item
Case 1: $\psi \in H^{s+\dl} (\T)$;

\item
Case 2: $\psi \notin H^{s+\dl} (\T)$.
\end{itemize}

Case 1:
We set
\[
\psi^{(\l)} (x) = \psi(x) + \frac 1\l \sum_{k \in 2^\N} k^{-s-\dl} (e^{i k x} + e^{-ikx})
\]
for $\l \in \N$.
It follows from $\psi \in H^{s+\dl}(\T)$ that
\begin{equation}
P_\pm \psi^{(\l)} = P_\pm (\psi^{(\l)} - \psi) + P_\pm \psi \notin H^{s+\dl}(\T)
\label{psiapp1}
\end{equation}
for any $\l \in \N$.

A simple calculation yields that
\begin{equation}
\| \psi^{(\l)} - \psi \|_{H^s}
= \bigg( \frac 2{\l^2}  \sum_{k \in 2^{\N}} k^{-2\dl} \bigg)^{\frac 12}
\les \frac 1\l
\label{psiapp2}
\end{equation}
for any $\l \in \N$.
Note that
\begin{align*}
&\bigg|
\int_\T \Im F_\be ( \psi^{(\l)},\dx \psi^{(\l)},\cj{\psi^{(\l)}},\cj{\dx \psi^{(\l)}} ) dx
- \int_\T \Im F_\be ( \psi,\dx \psi,\cj{\psi},\cj{\dx \psi} ) dx
\bigg|
\\
&\le
\big\| F_\be ( \psi^{(\l)},\dx \psi^{(\l)},\cj{\psi^{(\l)}},\cj{\dx \psi^{(\l)}} ) - F_\be ( \psi,\dx \psi,\cj{\psi},\cj{\dx \psi} ) \big\|_{L^2}
\\
&\le
C(\| \psi \|_{H^2})
\| \psi^{(\l)} - \psi \|_{H^2}.
\end{align*}
From \eqref{psiapp2} and $s>\frac 52$,
we obtain that
\[
\lim_{\l \to \infty} \int_\T \Im F_\be ( \psi^{(\l)},\dx \psi^{(\l)},\cj{\psi^{(\l)}},\cj{\dx \psi^{(\l)}} ) dx
= \int_\T \Im F_\be ( \psi,\dx \psi,\cj{\psi},\cj{\dx \psi} ) dx.
\]
With \eqref{psiapp1},
we can take $\phi = \psi^{(\l)}$ for some large $\l \in \N$.

Case 2:
We further consider the following three cases:

\begin{itemize}
\item
Case 2-1:
$P_+ \psi \notin H^{s+\dl}(\T)$ and $P_- \psi \notin H^{s+\dl}(\T)$;

\item
Case 2-2:
$P_+ \psi \in H^{s+\dl}(\T)$ and $P_- \psi \notin H^{s+\dl}(\T)$;

\item
Case 2-3:
$P_+ \psi \notin H^{s+\dl}(\T)$ and $P_- \psi \in H^{s+\dl}(\T)$.
\end{itemize}

Case 2-1:
We can take $\phi = \psi$.

Case 2-2:
We set
\[
\psi^{(\l)}_+ (x) = \psi(x) + \frac 1\l \sum_{k \in 2^\N} k^{-s-\dl} e^{i k x}
\]
for $\l \in \N$.
The same argument as in Case 1 shows that
we can take
$\phi = \psi^{(\l)}_+$ for some large $\l \in \N$.

Case 2-3:
This case is similarly handled as in Case 2-2.
\end{proof}

In the rest of this section,
we only consider a solution forward in time.
Namely, let $s>\frac 52$ and $u \in C([0,T]; H^s(\T))$ be a solution to \eqref{NLS}.

In Subsection \ref{SUBSEC:apriori1} below,
we prove an a priori estimate,
which plays a crucial role in the proof of Theorem \ref{thm:NE2}.
In Subsection \ref{subsec:nonex2},
we will prove Theorem \ref{thm:NE2}.

\subsection{An a priori estimate}
\label{SUBSEC:apriori1}

Similarly as in Subsection \ref{SUBSEC:energy},
we set
\[
v = \dx u, \quad
w = \dx v.
\]
For simplicity,
we write
\[
\Ta = F(u, v, \cj{u}, \cj{v}),
\quad
\Ta_\al = F_\al (u, v, \cj{u}, \cj{v}),
\quad
\dots.
\]
From \eqref{ueps}--\eqref{weps} with $\eps=0$,
$u$,$v$, and $w$ satisfy
\begin{align*}
&\dt u + i \dx^2 u
= \Ta,
\\
&\dt v + i \dx^2 v
=
R_1,
\\
&\dt w + i \dx^2 w
=
\Ta_{\be} \dx w
+
\Ta_{\cj \be} \cj{\dx w}
+
R_2,
\end{align*}
where $R_1$ and $R_2$ are polynomials in
$u,v,w, \cj{u}, \cj{v}, \cj{w}$.

Define
\begin{align*}
\Ld
&:=
-\frac {i}2 \dx^{-1} \Ta_\be,
&
W
&:=
e^{-\Ld} w.
\end{align*}
It follows from \eqref{Weps} with $\eps=0$ that
$W$ satisfies
\begin{equation}
\dt W + i \dx^2 W
=
(P_0 \Ta_\be) \dx W 
+
\Ta_{\cj \be} e^{-\Ld + \cj \Ld} \cj{\dx W}
+
\wt R,
\label{NLSW}
\end{equation}
where $\wt R$ is harmless.

Set
\begin{equation}
\ft \fw(t,k)
=
\exp \biggl(-i \bigg( k^2t + k \int^t_0 \Re P_0 \Ta_\be(t') dt' \bigg) \biggr) \ft W(t,k).
\label{fwga}
\end{equation}
With \eqref{NLSW},
$\ft \fw$ satisfies
\begin{equation}
\begin{aligned}
&\dt \ft \fw (t,k)
\\
&=
- (\Im P_0 \Ta_\be (t)) k \ft \fw (t,k)
\\
&\quad
+ \exp \biggl(-i \bigg( k^2t + k \int^t_0 \Re P_0 \Ta_\be(t') dt' \bigg) \biggr)
\Ft \big[ \Ta_{\cj \be} e^{-\Ld + \cj \Ld} \cj{\dx W} \big] (t, k)
\\
&\quad
+
\exp \biggl(-i \bigg( k^2t + k \int^t_0 \Re P_0 \Ta_\be(t') dt' \bigg) \biggr) \ft{\wt R}(t,k)
\\
&=:
- (\Im P_0 \Ta_\be (t)) k \ft \fw (t,k)
+
\Nl_1 (t,k)
+
\Nl_2 (t,k).
\label{ftfw1}
\end{aligned}
\end{equation}
We decompose $\Nl_1$ into two parts as follows:
\begin{equation}
\begin{aligned}
\Nl_1(t,k)
&=
i \sum_{k_1+k_2=k} e^{-i(k^2 + {k_2}^2)t} e^{-i k_1 \int^t_0 \Re P_0 \Ta_\be(t') dt'}
\\
&\hspace*{80pt}
\times
\Ft \big[ \Ta_{\cj \be} e^{-\Ld + \cj \Ld} \big] (t, k_1) k_2 \cj {\ft \fw (t,-k_2)}
\\
&=
i
\biggl(\sum_{\substack {k_1+k_2=k \\ |k_1| \ge \frac{|k_2|}2}}
+
\sum_{\substack {k_1+k_2=k \\ |k_1| < \frac{|k_2|}2}}
\biggr)
e^{-i(k^2 + {k_2}^2)t} e^{-i k_1 \int^t_0 \Re P_0 \Ta_\be(t') dt'}
\\
&\hspace*{100pt}
\times \Ft \big[ \Ta_{\cj \be} e^{-\Ld + \cj \Ld} \big] (t, k_1) k_2 \cj{\ft \fw (t,-k_2)}
\\
&=:
\Nl_{1,1} (t,k)
+
\Nl_{1,2} (t,k).
\end{aligned}
\label{decF1}
\end{equation}
We further decompose $\Nl_{1,2}$ into two parts:
\begin{equation}
\begin{aligned}
\Nl_{1,2}(t,k)
&= \dt \Ml(t,k) + \Kl(t,k),
\\
\Ml (t,k)
&:=
- \sum_{\substack{k_1 + k_2=k \\ |k_1| < \frac{|k_2|}2}}
\frac{e^{-i(k^2 + {k_2}^2)t}}{k^2 + k_2^2}
e^{-i k_1 \int^t_0 \Re P_0 \Ta_\be (t') dt'}
\\
&\hspace*{100pt}
\times
\Ft \big[ \Ta_{\cj \be} e^{-\Ld + \cj \Ld} \big] (t,k_1) k_2 \cj {\ft \fw (t,-k_2)},
\\
\Kl (t,k)
&:=
\sum_{\substack{k_1 + k_2=k \\ |k_1| < \frac{|k_2|}2}}
\frac{e^{-i(k^2 + {k_2}^2)t}}{k^2 + k_2^2}
\dt \biggl(e^{-i k_1 \int^t_0 \Re P_0 \Ta_\be (t') dt'}
\\
&\hspace*{100pt}
\times
\Ft \big[ \Ta_{\cj \be} e^{-\Ld + \cj \Ld} \big] (t,k_1) k_2 \cj {\ft \fw (t,-k_2)} \biggr).
\end{aligned}
\label{decF2}
\end{equation} 
From \eqref{ftfw1}--\eqref{decF2},
we have
\begin{equation}
\begin{aligned}
\dt \ft \fw (t,k)
&=
- (\Im P_0 \Ta_\be (t)) k \ft \fw (t,k)
+ 
\Nl_{1,1} (t,k)
\\
&\quad
+
\dt \Ml (t,k)
+
\Kl (t,k)
+
\Nl_2 (t,k).
\end{aligned}
\label{ftfw2}
\end{equation}
The following is the main estimate in this subsection.

\begin{proposition}
\label{prop:FGH}
Let $s>\frac52$.
With the notation above,
there exists a constant $C( \| u \|_{L_T^\infty H^s})>0$ such that the following estimates hold:
\begin{align}
\sup_{k \in \Z} \jb{k}^{s-2} \big( |\Nl_{1,1}(t,k)| + |\Nl_2(t,k)| \big)
&\le
C( \| u \|_{L_T^\infty H^s}),
\label{F1F2}
\\
\sup_{k \in \Z} \jb{k}^{s-1} |\Ml (t,k)| 
&\le
C( \| u \|_{L_T^\infty H^s}),
\label{Gest}
\\
\sup_{k \in \Z} \jb{k}^{s-2} |\Kl (t,k)|
&\le
C( \| u \|_{L_T^\infty H^s})
\label{Hest}
\end{align}
for $t \in [0,T]$.
\end{proposition}

\begin{proof}
For simplicity, we suppress the time dependence in this proof.
First, we prove \eqref{F1F2}. 
Note that $k_1,k_2 \in \Z$ with $|k_1| \ge \frac{|k_2|}2$ implies that $\jb{k_1+k_2} \les \jb{k_1}$.
It follows from \eqref{decF1} and H\"older's inequality that
\begin{equation}
\begin{aligned}
\jb{k}^{s-2} |\Nl_{1,1}(k)|
&\les
\sum_{\substack {k_1 + k_2 = k \\ |k_1| \ge \frac{|k_2|}2}}
\jb{k_1}^{s-1} \big| \Ft \big[ \Ta_{\cj \be} e^{-\Ld + \cj \Ld} \big] (k_1) \big|  |\ft \fw (-k_2)| 
\\
&\les
\big\| \Ta_{\cj \be} e^{-\Ld + \cj \Ld} \big\|_{H^{s-1}} \|\fw\|_{L^2}
\\
&\les
\| \Ta_{\cj \be}\|_{H^{s-1}} \big\| e^{-\Ld + \cj \Ld} \big\|_{H^{s-1}}  \|\fw\|_{L^2}
\\
&\le
C( \| u \|_{L_T^\infty H^s}).
\label{F11}
\end{aligned}
\end{equation} 
Since $\wt R$ is harmless, we have 
\begin{equation}
\jb{k}^{s-2} |\Nl_2 (k)|
=
\jb{k}^{s-2} |\Ft [\wt R] (k)|
\les
\|\wt R\|_{H^{s-2}}
\le
C( \| u \|_{L_T^\infty H^s}).
\label{F2}
\end{equation} 
Combining \eqref{F11} and \eqref{F2}, we get \eqref{F1F2}.

Next, we consider \eqref{Gest}.
By \eqref{decF2}, we obtain that
\begin{align*}
\jb{k}^{s-1} |\Ml (k)|
&\les
\sum_{\substack{k_1 + k_2=k \\ |k_1| < \frac{|k_2|}2}}
\big| \Ft \big[ \Ta_{\cj \be} e^{-\Ld + \cj \Ld} \big] (k_1) \big|
\jb{k_2}^{s-2} |\ft \fw(-k_2)|
\\
&\les
\big\| \Ta_{\cj \be} e^{-\Ld + \cj \Ld} \big\|_{L^2} \|\fw\|_{H^{s-2}}
\\
&\les
\| \Ta_{\cj \be}\|_{H^1} \big\| e^{-\Ld + \cj \Ld}  \big\|_{H^1} \|\fw\|_{H^{s-2}}
\\
&\le
C( \| u \|_{L_T^\infty H^s}),
\end{align*}
which shows \eqref{Gest}.

Finally, we consider \eqref{Hest}.
It follows from \eqref{decF2} that
\begin{equation}
\begin{aligned}
&\jb{k}^{s-2} |\Kl (k)|
\\
&\les
\sum_{\substack{k_1 + k_2=k \\ |k_1| < \frac{|k_2|}2}}
\biggl(
\jb{k_1} |\Re P_0 \Ta_\be |
\big| \Ft \big[ \Ta_{\cj \be} e^{-\Ld + \cj \Ld} \big] (k_1) \big| \jb{k_2}^{s-3}|\ft \fw(-k_2)|
\\
&\hspace*{60pt}
+
\jb{k_1}^{-1}
\big| \dt \Ft \big[ \Ta_{\cj \be} e^{-\Ld + \cj \Ld} \big] (k_1) \big|
\jb{k_2}^{s-2} |\ft \fw(-k_2)|
\\
&\hspace*{60pt}
+
\big| \Ft \big[ \Ta_{\cj \be} e^{-\Ld + \cj \Ld} \big] (k_1) \big| \jb{k_2}^{s-3} |\dt \ft \fw(-k_2)|
\biggr)
\\
&=:
\Kl_1(k) + \Kl_2(k) + \Kl_3(k).
\label{pfHest}
\end{aligned}
\end{equation}

First, we estimate $\Kl_1$.
H\"older's inequality and Proposition \ref{prop:bili} yield that
\begin{equation}
\begin{aligned}
|\Kl_1(k)|
&\les
\| \Ta_\be \|_{L^2}
\sum_{\substack{k_1 + k_2=k \\ |k_1| < \frac{|k_2|}2}}
\jb{k_1}
\big| \Ft \big[ \Ta_{\cj \be} e^{-\Ld + \cj \Ld} \big] (k_1) \big| 
\jb{k_2}^{s-3}|\ft \fw(-k_2)|
\\
&\les
\| \Ta_\be \|_{L^2}
\big\| \Ta_{\cj \be} e^{-\Ld + \cj \Ld} \big\|_{H^1} \|\fw\|_{H^{s-3}}
\\
&\les
\| \Ta_\be \|_{L^2}
\| \Ta_{\cj \be}\|_{H^1} \big\| e^{-\Ld + \cj \Ld} \big\|_{H^1}  \|\fw\|_{H^{s-3}}
\\
&
\le C( \| u \|_{L_T^\infty H^s}).
\label{estH1}
\end{aligned}
\end{equation}
Second, we consider $\Kl_2$.
H\"older's inequality and Proposition \ref{prop:bili} imply that
\begin{equation}
\begin{aligned}
&|\Kl_2(k)|
\\
&\les
\big\| \dt \big( \Ta_{\cj \be} e^{-\Ld + \cj \Ld} \big) \big\|_{H^{-1}} \| \fw \|_{H^{s-2}}
\\
&\les
\big( \big\| \dt \Ta_{\cj \be} \cdot e^{-\Ld + \cj \Ld} \big\|_{H^{-1}} + \big\| \Ta_{\cj \be} \dt \big( e^{-\Ld + \cj \Ld} \big) \big\|_{H^{-1}} \big) \|\fw\|_{H^{s-2}}
\\
&\les
\big( \| \dt \Ta_{\cj \be} \|_{H^{-1}} \big \| e^{-\Ld + \cj \Ld} \big\|_{H^1}
+ \| \Ta_{\cj \be} \|_{H^1} \big\| \dt \big( e^{-\Ld + \cj \Ld} \big) \big\|_{H^{-1}} \big) \|\fw\|_{H^{s-2}}
.
\end{aligned}
\label{estH2a}
\end{equation}
By \eqref{Tabbep} with $\eps=0$,
we have
\[
\begin{aligned}
\dt \Ta_{\cj \be}
&=
-i \dx (\Ta_{\be \cj \be}  w - \Ta_{\cj \be \cj \be} \cj{w})
+ R_3,
\end{aligned}
\]
where $R_3$ is a polynomial in $u,v,w, \cj u, \cj v, \cj w$.
Then,
it holds that
\begin{align}
\| \dt \Ta_{\cj \be} \|_{H^{-1}}
&\le
\| \Ta_{\be \cj \be} w - \Ta_{\cj \be \cj \be} \cj{w} \|_{L^2}
+ \| R_3 \|_{H^{-1}}
\notag
\\
&\le C( \| u \|_{L_T^\infty H^s}),
\label{estH2ba}
\\
\big\| \dt \big( e^{-\Ld + \cj \Ld} \big) \big\|_{H^{-1}}
&\les
\big\| e^{-\Ld + \cj \Ld} \big\|_{H^1} \| -\dt \Ld + \cj{\dt \Ld} \|_{H^{-1}}
\notag
\\
&\le C( \| u \|_{L_T^\infty H^s}).
\label{estH2bb}
\end{align}
From \eqref{estH2a}--\eqref{estH2bb}, we obtain that
\begin{equation}
|\Kl_2(k)|
\le C( \| u \|_{L_T^\infty H^s}).
\label{estH2}  
\end{equation}

Third, we estimate $\Kl_3$.
It follows from \eqref{pfHest} that
\begin{align*}
|\Kl_3(k)|
&\les
\sum_{k_1 \in \Z} \big| \Ft \big[ \Ta_{\cj \be} e^{-\Ld + \cj \Ld} \big] (k_1) \big|
\sup_{k_2 \in \Z} \jb{k_2}^{s-3} |\dt \ft \fw(-k_2)|
\\
&\les
\big\| \Ta_{\cj \be} e^{-\Ld + \cj \Ld} \big\|_{H^1}
\sup_{k_2 \in \Z} \jb{k_2}^{s-3} |\dt \ft \fw(-k_2)|.
\end{align*}
The first equality in \eqref{ftfw1}, Proposition \ref{prop:bili}, and $s>\frac 52$ yield that
\begin{align*}
&\jb{k_2}^{s-3} |\dt \ft \fw(-k_2)|
\\
&\les
\jb{k_2}^{s-2} |P_0 \Ta_\be(t)| |\ft \fw(-k_2)|
\\
&\quad
+ \jb{k_2}^{s-3}
\big| \Ft \big[ \Ta_{\cj \be} e^{-\Ld + \cj \Ld} \cj{\dx W} \big] (-k_2) \big|
+ \jb{k_2}^{s-3}
|\ft{\wt R}(-k_2)|
\\
&\les
\| \Ta \|_{L^2} \| \fw \|_{H^{s-2}}
+
\big\| \Ta_{\cj \be} e^{-\Ld + \cj \Ld} \cj{\dx W} \big\|_{H^{s-3}}
+ \| \wt R \|_{H^{s-3}}
\\
&\les
\| \Ta \|_{L^2} \| \fw \|_{H^{s-2}}
+
\big\| \Ta_{\cj \be} e^{-\Ld + \cj \Ld} \big\|_{H^{s-2}} \| \dx W \|_{H^{s-3}}
+ \| \wt R \|_{H^{s-3}}
\\
&\le
C( \| u \|_{L_T^\infty H^s}).
\end{align*}
Hence, we have that
\begin{equation}
|\Kl_3(k) | \le
C( \| u \|_{L_T^\infty H^s}).
\label{estH3}
\end{equation}
From \eqref{pfHest}, \eqref{estH1}, \eqref{estH2}, and \eqref{estH3},
we obtain \eqref{Hest}.
This concludes the proof.
\end{proof}

\subsection{Proof of Theorem \ref{thm:NE2}}
\label{subsec:nonex2}

We use the same notations as in Subsection \ref{SUBSEC:apriori1}.
We only consider the case (i) in Theorem \ref{thm:NE2},
since the case (ii) follows from a straightforward modification.
In what follows, we assume that
\begin{equation}
\Im P_0 \Ta_\be(0) >0.
\label{NI:phim2}
\end{equation}

From \eqref{ftfw2}, we have
\begin{align*}
\ft \fw(t,k) 
&=
e^{- k\int^t_0 \Im P_0 \Ta_\be(\tau) d\tau} \ft \fw(0)
\\
&\quad
+ \Big( \Ml(t,k) - e^{-k\int^t_0 \Im P_0 \Ta_\be(\tau) d\tau} \Ml(0,k) \Big) 
\\
&\quad
+
\int^t_0 e^{-k \int^t_{t'} \Im P_0 \Ta_\be(\tau) d\tau} \big( k (\Im P_0 \Ta_\be(t')) \Ml (t',k)
\\
&\hspace*{100pt}
+ \Nl_{1,1} (t',k) +\Nl_2 (t',k) + \Kl (t',k) \big) dt'
\end{align*}
for $t \in [0,T]$ and $k \in \Z$.
It holds that
\[
\begin{aligned}
\ft \fw(0,k)
&=
e^{k\int^t_0 \Im P_0 \Ta_\be(\tau) d\tau} \ft \fw(t,k)
\\
&\quad
- \Big( e^{k\int^t_0 \Im P_0 \Ta_\be(\tau) d\tau} \Ml(t,k) - \Ml(0,k) \Big)
\\
&\quad
-
\int^t_0 e^{k \int^{t'}_0 \Im P_0 \Ta_\be(\tau) d\tau}
\big( k (\Im P_0 \Ta_\be(t')) \Ml (t',k)
\\
&\hspace*{70pt}
+ \Nl_{1,1} (t',k) +\Nl_2 (t',k) + \Kl (t',k) \big) dt'.
\end{aligned}
\]
By Proposition \ref{prop:FGH}, we obtain that
\begin{equation}
\begin{aligned}
\jb{k}^{s-1} |\ft \fw(0,k)|
&\le
C( \| u \|_{L_T^\infty H^s})
\bigg(
\jb{k} e^{k\int^t_0 \Im P_0 \Ta_\be(\tau) d\tau}
+ 1
\\
&\hspace*{90pt}
+
\jb{k} \int^t_0 e^{k \int^{t'}_0 \Im P_0 \Ta_\be(\tau) d\tau} dt'
\bigg)
\label{kft0}
\end{aligned}
\end{equation}
for $t \in [0,T]$ and $k \in \Z$.

From the continuity of $P_0 \Ta_\be$ and \eqref{NI:phim2},
there exists $ T^{\ast} \in (0,T]$ such that
\begin{equation}
\int^{t}_0 \Im P_0 \Ta_\be (\tau) d\tau
\ge
\frac {\Im P_0 \Ta_\be (0)}2 t
\label{intM}
\end{equation}
for any $t\in [0,T^{\ast}]$.
Then,
we have
\begin{equation}
\begin{aligned}
\int^{T^{\ast}}_0 e^{k \int^{t'}_0 \Im P_0 \Ta_\be(\tau) d\tau} dt'
&\le
\int^{T^{\ast}}_0 e^{k \frac{\Im P_0 \Ta_\be(0)}2 t'} dt'
\\
&= \frac 2{k \Im P_0 \Ta_\be(0)} \Big( e^{k \frac{\Im P_0 \Ta_\be(0)}2 T^\ast} - 1 \Big)
\\
&\le
\frac 2{|k| \Im P_0 \Ta_\be(0)}
\end{aligned}
\label{eintM}
\end{equation}
for $k <0$.

By \eqref{intM} and the fact that
\[
\sup_{x>0} (1+x)e^{-ax} 
\le
\max(1,a)
\]
for $a>0$,
we have
\begin{equation}
\begin{aligned}
\jb{k} e^{k\int^{T^\ast}_0 \Im P_0 \Ta_\be(\tau) d\tau}
&\le (1+|k|) e^{-|k| \frac{\Im P_0 \Ta_\be(0)}2 T^\ast}
\\
&\le \max \Big( 1, \frac{\Im P_0 \Ta_\be(0)}2 T^\ast \Big)
\end{aligned}
\label{estM00}
\end{equation}
for $k<0$.
It follows from \eqref{kft0} with $t=T^\ast$, \eqref{eintM}, and \eqref{estM00} that
\[
\sup_{k <0}
\jb{k}^{s-1} |\ft \fw(0,k)|
\le
C( \| u \|_{L_T^\infty H^s}).
\]
Hence,
we have
\begin{align*}
\| P_- \fw(0) \|_{H^r}^2
&\le
\sum_{k<0} \jb{k}^{2r} |\fw(0,k)|^2
\le
C( \| u \|_{L_T^\infty H^s})
\sum_{k<0} \jb{k}^{2(r-s+1)}
\\
&\le
C( r, \| u \|_{L_T^\infty H^s})
\end{align*}
for $r<s-\frac 32$.
It follows from Proposition \ref{prop:bili2}, Lemma \ref{lem:bili3}, and \eqref{fwga} that
\begin{align*}
\| P_- \phi \|_{H^{r+2}}
&= 
\| P_- ( e^{\Ld} W)(0) \|_{H^r}
\\
&\le
\| P_- ( e^{\Ld} P_- W)(0) \|_{H^r}
+ 
\| P_- ( e^{\Ld} P_0 W)(0) \|_{H^r}
\\
&\hspace*{120pt}
+ 
\| P_- ( e^{\Ld} P_+ W)(0) \|_{H^r}
\\
&\les
\| e^{\Ld(0)} \|_{H^{\frac 12 + \eps}} \|  P_- \fw(0) \|_{H^r}
+
\| e^{\Ld(0)} \|_{H^r} \| W(0) \|_{H^{\frac 12+\eps}}
\\
&\le
C( r, \| u \|_{L_T^\infty H^s})
\end{align*}
for $r \in [0,s-\frac 32)$ and $\eps \in (0, s-\frac 52)$.
By taking $r = s-2+\dl$ for $\dl \in (0,\frac 12)$,
this shows (i) in Theorem \ref{thm:NE2}.

\mbox{}

\noindent
{\bf 
Acknowledgements.}
M.O.~was supported by JSPS KAKENHI Grant number JP23K03182.
The authors would like to thank the anonymous referee for helpful comments.

\end{document}